\documentclass[10pt,reqno]{amsart}
\usepackage{amsmath,amscd,amsfonts,amsthm,amssymb,latexsym}

\theoremstyle{plain}
   \newtheorem{theorem}{Theorem}[section]
   \newtheorem{proposition}[theorem]{Proposition}
   \newtheorem{lemma}[theorem]{Lemma}
   \newtheorem{corollary}[theorem]{Corollary}
   \newtheorem{problem}{Problem}
   
   \newtheorem{question}{Question}
\theoremstyle{definition}
   
   \newtheorem{example}{Example}
   
\theoremstyle{remark}
   \newtheorem{remark}{Remark}[section]

\author[J.~Borcea]{Julius Borcea}
\address{Department of Mathematics, Stockholm University, SE-106 91 Stockholm,
Sweden}
\email{julius@math.su.se}

\author[P.~Br\"and\'en]{Petter Br\"and\'en}
\address{Department of Mathematics, Royal Institute of Technology, 
SE-100 44 Stockholm, Sweden}
\email{pbranden@math.kth.se}
       

\thanks{The first author was partially supported by the 
Swedish Research Council and the Crafoord Foundation. The second author was 
partially supported by the G\"oran Gustafsson Foundation.}

\numberwithin{equation}{section}

\newcommand{\NN}{\mathbb{N}}

\newcommand{\bH}{\mathbb{H}}

\newcommand{\NV}{\mathcal{N}}

\newcommand{\HH}{\mathcal{H}}

\newcommand{\HHH}{\overline{\mathcal{H}}}

\newcommand{\RR}{\mathbb{R}}
\newcommand{\CC}{\mathbb{C}}
\newcommand{\DD}{\mathbb{D}}
\newcommand{\KK}{\mathbb{K}}

\newcommand{\sym}{\mathfrak{S}}
\newcommand{\si}{\mathfrak{s}}
\newcommand{\Trp}{\mathcal{A}}
\newcommand{\te}{\theta}

\renewcommand{\Im}{{\rm Im}}
\renewcommand{\Re}{{\rm Re}}

\def\newop#1{\expandafter\def\csname #1\endcsname{\mathop{\rm
#1}\nolimits}}

\newop{diag}
\newop{supp}
\newop{Proj}
\newop{JP}
\newop{Hom}
\newop{Sym}
\newop{Symb}
\newop{PolP}
\newop{PolO}
\newop{glb}
\newop{MAP}
\newop{MOD}
\newop{hom}

\begin{document}

\title[The Lee-Yang and P\'olya-Schur Programs. II.]{The Lee-Yang 
and P\'olya-Schur Programs. II. \\ 
Theory of Stable Polynomials and Applications}

\begin{abstract} 
In the first part of this series we characterized all linear 
operators on 
spaces of multivariate polynomials preserving the property of being  
non-vanishing 
 in products of open circular domains. For such sets 
this completes the multivariate generalization of  the classification program 
initiated 
by P\'olya-Schur for univariate real polynomials. 
We build on these classification 
theorems to develop here a theory of multivariate stable polynomials. 
Applications and examples show that this theory provides 
a natural framework for dealing in a uniform way with 
Lee-Yang type problems in statistical mechanics, combinatorics, and geometric
function theory in one or several variables. In particular, we answer a 
question of Hinkkanen on multivariate apolarity.
\end{abstract}

\maketitle

\tableofcontents

\section*{Introduction}\label{s-1}

In two seminal papers from 1952 \cite{LY,LY1} Lee and Yang proposed the 
program 
of analyzing phase transitions in terms of zeros of the partition function 
and proved a celebrated theorem locating the zeros of the partition 
function of the ferromagnetic Ising model on the imaginary axis in the 
complex magnetic plane. This theorem has since been proved and generalized 
in many ways by e.g.~Asano, Fisher, Newman, Ruelle,
Lieb-Sokal, Biskup {\em et al}, etc; see \S \ref{ss-94} and references 
therein. Nevertheless, the Lee-Yang theorem seems to have retained an aura of 
mystique. In his 1988 Gibbs lecture \cite{ruelle2} Ruelle proclaimed: ``I 
have called this beautiful result a failure because, while it has important 
applications in physics, it remains at this time isolated in mathematics.'' 
Ruelle's statement was apparently motivated by the fact that 
the Lee-Yang theorem also inspired speculations about possible statistical 
mechanics models underlying the zeros of Riemann or Selberg 
zeta functions and the Weil conjectures \cite{hink,new-R,ruelle2} but  
``the miracle has not happened'' \cite{ruelle2}.

Recently, Lee-Yang like problems and techniques have appeared in various 
mathematical contexts such as combinatorics, complex analysis, matrix theory 
and probability theory  
\cite{AC,BBS1,BBS2,BBS3,BBCV,BBL,Br1,car-F,COSW,fisk,HV,hink,ruelle1,sok,W2}. 
The past decade has also been marked by important developments on 
other aspects of phase transitions, conformal invariance, percolation 
theory \cite{KOS,LSW,SW}. However, as Hinkkanen noted in \cite{hink}, the 
power in the ideas behind the Lee-Yang theorem has not yet 
been fully exploited: ``It seems that the theory of polynomials, linear in 
each variable, that do not have zeros in a given multidisk or a more general 
set, has a long way to go, and has so far unnoticed connections to various 
other concepts in mathematics.''

In this paper we show that the Lee-Yang theorem and the mathematics 
around it are intimately connected with the dynamics of zero loci of 
multivariate polynomials under linear transformations and 
Problems \ref{prob1}--\ref{prob2} below. As we point out in 
\S \ref{ss-94}, such connections have been implicitly noted in essentially 
all known proofs and extensions of the Lee-Yang theorem. For instance, 
Lieb and  Sokal \cite{LS} reduced the at the time best Lee-Yang theorem, 
due to  Newman \cite{new2}, to the following statement: 
if $P, Q \in \CC[z_1,\ldots, z_n]$ are 
polynomials which are non-vanishing when all variables are in the open right 
half-plane, then the polynomial 
$
P(\partial / \partial z_1, \ldots, \partial / \partial z_n)
Q(z_1,\ldots, z_n)
$ 
also has this property unless it is identically zero (Theorem \ref{th-LS}). 
Thus, to better understand Lee-Yang type theorems 
one is naturally led to consider the problems of 
describing linear operators on polynomial spaces that preserve the property 
of being non-vanishing when the variables are in prescribed 
subsets of $\CC^n$.

Let us formulate these problems explicitly as in \cite{BB-I}. 
Given an integer $n\ge 1$ and $\Omega \subset \CC^n$ we say 
that $f\in\CC[z_1,\ldots,z_n]$ is 
$\Omega$-{\em stable}
if $f(z_1,\ldots,z_n)\neq 0$ whenever $(z_1,\ldots,z_n)\in\Omega$. 
A $\KK$-linear operator 
$T:V\to \KK[z_1,\ldots,z_n]$, where $\KK=\RR$ or $\CC$ and 
$V$ is a subspace of $\KK[z_1,\ldots,z_n]$, is said to {\em preserve} 
$\Omega$-{\em stability} if for any 
$\Omega$-stable polynomial $f\in V$ the 
polynomial $T(f)$ is either 
$\Omega$-stable or $T(f)\equiv 0$.
For $\kappa=(\kappa_1,\ldots,\kappa_n) \in \NN^n$ let 
$\KK_\kappa[z_1,\ldots, z_n]
=\{f\in \KK[z_1,\ldots, z_n]:\deg_{z_i}(f)\le \kappa_i,1\le i\le n\}$, 
where $\deg_{z_i}(f)$ is the degree of $f$ in $z_i$. By slight abuse of 
terminology, if $\Psi\subset \CC$ and 
$\Omega=\Psi^n$ then $\Omega$-stable polynomials will also be referred to as 
$\Psi$-stable.

\begin{problem}\label{prob1}
Characterize all linear operators 
$T:\KK_\kappa[z_1,\ldots, z_n]\to \KK[z_1,\ldots,z_n]$ that preserve 
$\Omega$-stability for a given set $\Omega\subset\CC^n$ and 
$\kappa\in\NN^n$.
\end{problem}

\begin{problem}\label{prob2}
Characterize all linear operators 
$T:\KK[z_1,\ldots, z_n]\to \KK[z_1,\ldots,z_n]$ that preserve 
$\Omega$-stability, where $\Omega$ are prescribed subsets of 
$\CC^n$.
\end{problem}

In physics 
\cite{sco-sok,sok} it is useful to distinguish between hard-core pair 
interactions (subject to constraints, e.g.~the maximum degree of a 
graph) and soft-core pair interactions (essentially constraint-free). 
By analogy with this dichotomy, one may say that results pertaining to 
Problem \ref{prob1} are ``hard'' or ``algebraic'' (bounded degree) while 
those for Problem \ref{prob2} are ``soft'' or ``transcendental''  
(unbounded degree), cf.~\cite{BB-I}.

For $n=1$, $\KK=\RR$, and $\Omega=\{z\in\CC:\Im(z)>0\}$ 
Problems \ref{prob1}--\ref{prob2} amount to classifying linear operators that 
preserve the set of real polynomials with all real zeros. This question has 
a distinguished history that goes back to 
Hermite, Laguerre, Hurwitz and P\'olya-Schur, see \cite{BBCV} and references 
therein. In particular, in \cite{PS} P\'olya and Schur characterized all 
diagonal operators with this property, 
%(so-called multiplier sequences), 
which led to a rich subsequent literature on this subject 
\cite{BBS1,BBS3,CC1,CCS,CC2,cso,fisk,iserles1,Le,M,pol-coll,RS,schur}.
However, it was not until very recently that full solutions to this question 
-- and, more generally, to Problems \ref{prob1}--\ref{prob2} for $n=1$ and any 
open circular domain $\Omega$ -- were obtained in \cite{BBS1}. 
Quite recently, 
Problems \ref{prob1}--\ref{prob2} were solved in \cite{BB-I} whenever 
$\Omega=\Omega_1\times\cdots\times\Omega_n$ and the $\Omega_i$'s 
are open circular domains. For such sets these 
results complete the multivariate generalization of the classification 
program initiated by P\'olya-Schur \cite{PS}. They also go beyond 
e.g.~\cite{BBS1,BBS3} and have interesting consequences, as we will 
now see. 
 
In Part A we build on the classification theorems of 
\cite{BB-I} to develop a self-contained theory of multivariate stable 
polynomials. 
%We rely on a minimum of auxiliary results and establish the process we rely  provide proofs for establish . 
%it turns out that the results of \cite{BB-I} and this 
%theory are essentially obtained from two basic facts: 
%(multivariate versions of) Hurwitz's theorem on the ``continuity of zeros'' 
%and the maximum principle for harmonic functions.  
To begin with, in \S \ref{s-3} we study 
operators on multi-affine polynomials inspired by natural time 
evolutions (symmetric exclusion processes) for interacting 
particle systems \cite{Lig2}. We give a new simple proof of 
\cite[Theorem~4.20]{BBL} (see also \cite{Lig1}) stating that 
these operators preserve stability and extend it to all circular 
domains.
In \S \ref{s-gws} we use these symmetrization procedures to give a new proof 
of  the Grace-Walsh-Szeg\"o coincidence theorem that unlike most proofs 
known so far avoids (univariate) apolarity theory. 
In \S \ref{s-hilf} we establish a ``master composition theorem'' 
that extends to several variables all the classical Hadamard-Schur 
convolution results due 
to Schur-Mal\'o-Szeg\"o, Walsh, 
Cohn-Egerv\'ary-Szeg\"o, de Bruijn, etc \cite{CC1,M,RS}. 
This also generalizes the 
multivariate composition theorems based on the 
Weyl product \cite{BBS3} as well as Hinkkanen's 
theorem \cite{hink} 
and provides a 
unifying framework for results of this type. In \S \ref{s-hard-ps} we 
obtain ``hard'' multivariate   
generalizations of P\'olya-Schur's classification of multiplier sequences 
that extend the ``soft'' theorems of \cite{BBS3}.

As noted in \cite{rota-lect}, the concept of apolarity has a rich 
pedigree going all the way back to Apollonius and was much studied in 
invariant theory, umbral calculus, and algebraic geometry \cite{K-R,s-b}. In 
\cite{rota-lect} Rota adds: ``Grace's [apolarity] theorem is an instance 
of what might be called a sturdy theorem. For almost one hundred years it has 
resisted all attempts at generalization. Almost all known  results
about the distribution of zeros of [univariate] polynomials in the complex 
plane are corollaries of Grace's theorem.'' In \S \ref{s-apol} we  
establish Grace type theorems for 
 multivariate polynomials and provide an answer to 
 a question of Hinkkanen 
\cite[\S 5]{hink}. In 
\S \ref{ss-91} we prove ``hard'' Lieb-Sokal lemmas that sharpen the 
``soft'' ones in \cite{LS} (whose importance in the Lee-Yang program is 
explained in \S \ref{ss-newman}.) 

%To further illustrate the power 
%of the classification theorems obtained in \cite{BB-I}, we 
%show how these immediately imply the main results of \cite{BBS3} for partial 
%differential operators in the Weyl algebra. 

%Finally, extending the results 
%of \cite{AC,car}, in \S \ref{s-p-hilf} we prove a higher-dimensional 
%generalization of a theorem of P\'olya \cite[Hilfssatz II]{pol-riem}. 
%The latter was an inspirational device 
%-- via Marc Kac \cite[pp.~424--426]{pol-coll} -- for Lee-Yang's proof of their 
%theorem and has recently also found applications to Fourier transforms with all
%real zeros \cite{car-F}.

%The common conceptual reason behind our main theorems 
%turns out to be fairly simple: the property of preserving 
%$\Omega_1\times\cdots\times\Omega_n$-stability is closed under 
%operator compositions. This is a quite natural explanation for phenomena 
%pertaining to zero loci of polynomials (cf.~also the applications in Part II) 
%that have been described as being of a ``somewhat mysterious nature'' 
%(Ruelle \cite{ruelle1}) or ``remarkable by their lack of intuitive content'' 
%(Craven-Csordas \cite{CC1}).

In Part B we study statistical mechanical and combinatorial applications 
of the theory of stable polynomials developed in Part A and \cite{BB-I}. We 
show that the key steps in existing proofs and extensions of the 
Lee-Yang and Heilmann-Lieb theorems as well as various other theorems on graph 
polynomials follow in a simple and unified manner from this theory. 
These results 
are due to Asano \cite{As}, Ruelle \cite{ruelle-g1,ruelle-g2,ruelle5}, 
Newman \cite{new1,new2}, Lieb-Sokal \cite{LS}, Hinkkanen \cite{hink},
Choe et al \cite{COSW}, Wagner \cite{W2}. 

%Finally, in \S \ref{s-ide} we discuss related results and point out further 
%directions. 
%provide natural 
%mathematical contexts for studying physical phenomena (Lee-Yang's  
%approach to phase transitions, interacting particle systems),  
%which substantiates Hinkkanen's observation quoted 
%above and will hopefully lead to further interesting applications.

%\newpage

\section*{{\bf A. Theory of Multivariate $C$-Stable 
Polynomials}}

Let us first fix some of the notation that we will use throughout. Recall 
that the {\em support}  of a polynomial 
$f(z)=\sum_{\alpha \in \NN^n} a(\alpha) z^\alpha \in \CC[z_1,\ldots, z_n]$ 
is the set $\supp(f)=\{\alpha \in \NN^n : a(\alpha) \neq 0\}$, where 
$z=(z_1,\ldots,z_n)\in\CC^n$, $\alpha = (\alpha_1,\ldots, \alpha_n)\in 
\NN^n$, and 
$z^\alpha= z_1^{\alpha_1}\cdots z_n^{\alpha_n}$. 
%Let $d:=\max\{|\alpha|:\alpha\in\supp(f)\}$, where 
%$|\alpha|=\alpha_1+\ldots+\alpha_n$, be the total degree of $f$. 
%and define the
%{\em homogeneous part} of $f$ as 
%$\lim_{t\to\infty}t^{-d}f(tz_1,\ldots,tz_n)$. 
Set $[n]=\{1,\ldots,n\}$ and $(1^n)=(1,\ldots,1)\in\NN^n$. We say that $f$ is 
of degree at most $\kappa=(\kappa_1,\ldots,\kappa_n) \in \NN^n$ if 
$f\in \CC_\kappa[z_1,\ldots,z_n]$ (cf.~the introduction) and of degree 
$\kappa$ if $\deg_{z_i}(f)= \kappa_i$, $i\in [n]$. Polynomials in 
$\CC_{(1^n)}[z_1,\ldots, z_n]$ are called {\em multi-affine}.

We employ the usual partial order on $\NN^n$: 
if $\alpha=(\alpha_1,\ldots, \alpha_n)\in \NN^n$ and  
$\beta=(\beta_1,\ldots, \beta_n) \in \NN^n$ then 
$\alpha \leq \beta$ if $\alpha_i \leq \beta_i$ for all $i \in [n]$. 
Let $|\alpha|=\alpha_1+\ldots+\alpha_n$, $\alpha!= \alpha_1!\cdots \alpha_n!$, 
and 
$$
\binom \beta \alpha = \begin{cases} 
\frac {\beta!}{\alpha!(\beta-\alpha)!} \mbox{ if } \alpha \leq \beta, \\ 
0 \mbox{ otherwise. }
\end{cases} 
$$

The open unit disk is denoted by $\DD$ and open half-planes bordering on 
the origin  by  
$\bH_\theta = \{z \in \CC : \Im(e^{i\theta}z)>0\}$, where  
$\theta\in[0,2\pi)$. Note that $\bH_0$ is the 
open upper half-plane
%, which we will sometimes denote just by $H$, 
while 
$\bH_{\frac{\pi}{2}}$ is the open right half-plane. 
$\bH_0$-stable polynomials 
%in $\CC[z_1,\ldots,z_n]$ 
are referred to as 
{\em stable} polynomials and those with all real coefficients are
called {\em real stable}, cf.~\cite{BB-I}--\cite{BBL}. We denote the sets of 
stable, respectively
real stable polynomials in $n$ variables by $\HH_n(\CC)$, respectively 
$\HH_n(\RR)$. Polynomials in $\CC[z_1,\ldots,z_n]$ which are 
$\bH_{\frac{\pi}{2}}$-stable are said to be 
{\em weakly Hurwitz stable}. In \cite{COSW} these are termed 
polynomials with the {\em half-plane property}. The notions of 
$\bH_{\te}$-stability are 
equivalent modulo rotations for complex polynomials but this is not so 
for real polynomials. However, for real polynomials with non-negative 
coefficients \cite[Theorem 4.5]{BBL} yields the following hierarchy of 
half-plane properties: if such a polynomial is $\bH_0$-stable 
then it is $\bH_{\te}$-stable for any $\te\in[0,\pi]$.

\section{Symmetrization Procedures}\label{s-3}

The symmetric group on $n$ elements, $\sym_n$, acts on 
$\CC[z_1,\ldots, z_n]$ by permuting the variables: 
$\sigma(f)(z_1,\ldots, z_n)=f(z_{\sigma(1)}, \ldots, z_{\sigma(n)})$, 
$\sigma \in \sym_n$, $f \in \CC[z_1,\ldots, z_n]$. Define  
 the {\em symmetrization operator} $\Sym : \CC[z_1,\ldots, z_n] 
\rightarrow \CC[z_1,\ldots, z_n]$ by 
$$
\Sym(f) = \frac 1 {n!} \sum_{\sigma \in \sym_n} \sigma(f). 
$$
Clearly, $\Sym$ is a linear operator whose image consists of 
{\em symmetric polynomials}, that is, 
polynomials invariant under the action of $\sym_n$. 

Recall that a {\em circular domain} in $\CC$ is any open or closed disk, 
exterior of a disk, or half-plane.  The Grace-Walsh-Szeg\"o coincidence 
theorem is an important and very useful result on the geometry of 
polynomials, see, e.g., 
\cite{COSW,grace,M,RS,walsh}. 

\begin{theorem}[Grace-Walsh-Szeg\"o]\label{GWS}
Let $f\in\CC_{(1^n)}[z_1,\ldots, z_n]$ be a symmetric polynomial, $C$ be an 
open or closed circular domain, and $\xi_1, \ldots, \xi_n\in C$.  
Suppose further that either the total degree of $f$ equals 
$n$ or that $C$ is convex (or both). Then there exists at least one point 
$\xi \in C$ such that 
\begin{equation}\label{ssss}
f(\xi_1, \ldots, \xi_n)= f(\xi, \ldots, \xi). 
\end{equation}
\end{theorem}

This theorem was essential 
for proving the sufficiency part of the characterization of linear operators 
preserving 
$C$-stability in \cite{BB-I}. We will see here that Theorem~\ref{GWS} is 
actually a consequence of stronger 
(asymmetric) symmetrization procedures on stable polynomials which were used 
in \cite{BBL} to prove correlation inequalities for symmetric exclusion 
processes. More precisely, we will deduce Theorem~\ref{GWS} 
from the following result.

\begin{theorem}\label{symmetrize}
Let $C$ be an open or closed circular domain.
\begin{itemize}
\item[(a)] If $C$ is convex then the 
symmetrization operator $\Sym$ preserves $C$-stability on 
multi-affine polynomials, i.e., 
$\Sym : \CC_{(1^n)}[z_1,\ldots, z_n] 
\rightarrow \CC_{(1^n)}[z_1,\ldots, z_n]$ preserves $C$-stability. 
\item[(b)] If $C$ is non-convex and 
$f \in  \CC_{(1^n)}[z_1,\ldots, z_n]$ is $C$-stable and such that all 
variables are active in $f$ (i.e., $\partial f/\partial z_i\neq 0$, 
$i\in [n]$) then $\Sym(f)$ is $C$-stable. 
\end{itemize} 
\end{theorem}

\begin{remark}\label{max-deg}
If $C$ is non-convex and $f \in  \CC_{(1^n)}[z_1,\ldots, z_n]$ is $C$-stable 
then the condition in Theorem \ref{symmetrize} (b) that all variables are 
active is actually equivalent to the requirement that $f$ has total degree
 $n$, i.e., $\partial^n f / \partial z_1 \cdots \partial z_n \neq 0$. 
\end{remark}

\begin{remark}
One can easily construct examples showing that $\Sym$ does not preserve 
$C$-stability when acting on arbitrary (not multi-affine) polynomials.
\end{remark}

It is not difficult to prove Theorem~\ref{symmetrize} {\em assuming} the 
Grace-Walsh-Szeg\"o theorem. However, in \S \ref{s-gws} we will prove the 
latter {\em via} Theorem~\ref{symmetrize}. This will make the theory 
developed here and in \cite{BB-I} self-contained. 

We will derive Theorem~\ref{symmetrize} from the next result which was first 
proved in \cite[Theorem 4.20]{BBL}. From a 
physical viewpoint \cite{Lig2}, Proposition~\ref{asymsym} implies 
that stability is preserved by the natural time evolution of symmetric 
exclusion processes.

\begin{proposition}\label{asymsym}
Let $f \in \CC_{(1^n)}[z_1,\ldots, z_n]$, $C$ be an open or
closed circular domain, $0\leq p \leq 1$, and $\tau = (ij) \in \sym_n$ be 
a transposition. 
\begin{itemize}
\item[(a)] If $C$ is convex and  $f$ is $C$-stable then so is 
$pf + (1-p)\tau(f)$. 
\item[(b)] If $C$ is non-convex and $f$ is $C$-stable and depending on 
both $z_i$ and $z_j$ then $pf + (1-p)\tau(f)$ is also $C$-stable. 
\end{itemize}
\end{proposition}

Our proof of Proposition~\ref{asymsym} relies on the maximum principle for 
harmonic functions which we use to prove the following lemma. 
Another recent elementary proof of Proposition~\ref{asymsym} was 
independently given in \cite{Lig1}. 
%NOTE THAT $\partial / \partial w$ does not preserve $\bar{H}$-stability. 
%Take $f= -1+iz+wz$.
Let $\overline{\bH}_0$ be the closed upper half-plane.

\begin{lemma}\label{maximum}
Let $f(z,w)=a+bz+cw+dzw \in \CC[z,w]$ and define 
\begin{eqnarray*}
&&V_1(f)(x)=\Im(a\bar{c})+\Im(a\bar{d}+b\bar{c})x + \Im(b\bar{d})x^2,  \\
&&V_2(f)(x)=\Im(a\bar{b})+\Im(a\bar{d}+c\bar{b})x + \Im(c\bar{d})x^2. 
\end{eqnarray*}
Suppose that $d \neq 0$. 
\begin{enumerate}
\item If $f$ is $\overline{\bH}_0$-stable then $\Im(b/d) >0$ or $\Im(c/d) >0$. 
\item  If $f$ is $\overline{\bH}_0$-stable and $\Im(c/d) >0$ then 
$V_1(f)(x)> 0$ and $V_2(f)(x) \geq 0$ for all 
$x \in \RR$. 
\item  If $f$ is $\overline{\bH}_0$-stable and $\Im(b/d) >0$ then 
 $V_2(f)(x) > 0$ and $V_1(f)(x) \geq 0$ for all 
$x \in \RR$. 
\item  If  $\Im(b/d) >0$ and $\Im(c/d) >0$ then $f$ is 
$\overline{\bH}_0$-stable if and only if for some (and then any) 
$i\in \{1,2\}$ one has $V_i(f)(x)>0$ for all $x \in \RR$.
\end{enumerate} 
\end{lemma}

\begin{proof}
Since the partial derivative of a $\bH_0$-stable polynomial is 
$\bH_0$-stable or identically zero (if in doubt apply 
Theorem \ref{open-halfplane}) we have $\Im(c/d)\geq 0$ and $\Im(b/d) \geq 0$ 
if $f$ is $\overline{\bH}_0$-stable. If $\Im(c/d)=\Im(b/d)=0$ then the 
polynomial $d^{-1}f(z-c/d,w-b/d)=wz+(ad-bc)/d^2$ is 
$\overline{\bH}_0$-stable. This is a contradiction since 
$\{wz: w,z \in \overline{\bH}_0\}=\CC$. Hence 
$\Im(c/d) >0$ or $\Im(b/d) >0$ if $f$ is $\overline{\bH}_0$-stable. 

Assume that $\Im(b/d) >0$  and $\Im(c/d)>0$. Solving for $w$ in 
$f(z,w)=0$ we see that $f$ is $\overline{\bH}_0$-stable if and only if 
\begin{equation}\label{harm}
\Im(z)\geq 0 \quad \Longrightarrow \quad 
\rho(z):=\Im\!\left(\frac{a+bz}{c+dz}\right) > 0.
\end{equation}
Now, $\rho$ is a harmonic function in the half-plane 
$\{ z\in\CC : \Im(z)> -\Im(c/d)\}$ which contains $\overline{\bH}_0$. Let 
$K_r =\{ z \in \CC : \Im(z) \geq 0, |z|\leq r\}$. By the maximum principle 
the minimum of $\rho$ on $K_r$ is attained on the boundary of $K_r$.  
For real $x$ we have 
$$
\rho(x)=  
\frac {\Im(a\bar{c})+\Im(a\bar{d}+b\bar{c})x + \Im(b\bar{d})x^2}{|c+dx|^2}.
$$ 
Moreover, $\rho(z) \rightarrow \Im(b/d)>0$ as $z \rightarrow \infty$. 
Since the same arguments apply if one instead solves for $z$ in 
$f(z,w)=0$, this verifies (4). 

If just $\Im(c/d) >0$ we will still have $\rho(x) >0$ for all $x \in \RR$ and 
if only  $\Im(b/d)>0$ then  $\rho(x) >0$ for all 
$x \in \RR \setminus \{-c/d \}$. By symmetry in $z$ and $w$ this verifies 
(2) and (3) of the lemma.   
\end{proof}

Recall the multivariate version of Hurwitz' theorem on the 
``continuity of zeros'', cf. \cite[Footnote 3, p.~96]{COSW}.

\begin{theorem}[Hurwitz' theorem]\label{mult-hur}
Let $D$ be a domain (open connected set) in $\CC^n$ and suppose 
$\{f_k\}_{k=1}^\infty$ is a sequence of non-vanishing 
analytic functions on $D$ that converge to $f$ uniformly on compact 
subsets of $D$. Then $f$ is either 
non-vanishing on $D$ or else identically zero. 
\end{theorem}

To deal with discs and exteriors of discs we also need 
Lemmas~\ref{maxsupport} and \ref{translate} below -- which were proved 
in \cite[Lemmas 6.1 and 6.2]{BB-I} -- and Corollary \ref{inv}.

\begin{lemma}\label{maxsupport}
Let $\{C_i\}_{i=1}^n$ be a family of circular domains, 
$f \in \CC[z_1,\ldots, z_n]$ be of degree $\kappa\in\NN^n$, and 
$J \subseteq [n]$ a (possibly empty) set such that $C_j$ is the exterior of 
a disk whenever $j \in J$. Denote by $g$ 
be the polynomial in the variables $z_j$, $j \in J$, obtained by setting 
$z_i = c_i \in C_i$ arbitrarily for $i \notin J$. If $f$ is 
$C_1\times\cdots\times C_n$-stable then 
$\supp(g)$ has a unique maximal element $\gamma$ with respect to the 
standard partial order on $\NN^J$. Moreover, 
$\gamma$ is the same for all choices of $c_i \in C_i$, $i \notin J$. 
\end{lemma}

An immediate consequence of Lemma~\ref{maxsupport} is the following.

\begin{corollary}\label{inv}
Let $\kappa\in\NN^n$ and $I_\kappa : \CC_\kappa[z_1,\ldots, z_n] \rightarrow 
\CC_\kappa[z_1, \ldots, z_n]$ be the linear operator defined by 
$I_\kappa(z^\alpha)=z^{\kappa-\alpha}$, $\alpha\le \kappa$.  
Then $I_\kappa$ restricts to a bijection between the set of $\DD$-stable 
($\overline{\DD}$-stable) polynomials in 
$\CC_\kappa[z_1,\ldots, z_n]$ and the set of 
$\CC\setminus \overline{\DD}$-stable 
($\CC\setminus {\DD}$-stable) polynomials in 
$\CC_\kappa[z_1,\ldots, z_n]$ of degree $\kappa$. 
\end{corollary}

%\begin{notation}\label{not-nv}
If $\{C_i\}_{i=1}^n$ is a family of circular domains and 
$\kappa=(\kappa_1,\ldots,\kappa_n) \in \NN^n$ we let 
$\NV_\kappa(C_1,\ldots,C_n)$ be the set of 
$C_1 \times \cdots \times C_n$-stable polynomials in  
 $\CC_\kappa[z_1,\ldots,z_n]$ that have degree $\kappa_j$ in 
$z_j$ whenever $C_j$ is non-convex. Note that if all $C_j$ are convex then 
$\NV_\kappa(C_1,\ldots,C_n)$ consists of all 
$C_1 \times \cdots \times C_n$-stable polynomials in  
$\CC_\kappa[z_1,\ldots,z_n]$. 
%\end{notation}

Recall that a {\em M\"obius transformation} is a bijective 
conformal map of the extended complex plane $\widehat{\CC}$ given by 
\begin{equation}\label{mobius}
\phi(\zeta)= \frac {a\zeta+b}{c\zeta+d}, \quad a,b,c,d \in \CC, ad-bc =1. 
\end{equation}
Note that one usually has the weaker requirement $ad-bc \neq 0$ but 
this is equivalent to \eqref{mobius} which proves to be 
more convenient. 

\begin{lemma}\label{translate}
Suppose that $C_1, \ldots, C_n, D_1,\ldots, D_n$ are open circular domains and 
$\kappa=(\kappa_1,\ldots,\kappa_n) \in \NN^n$. Then there are M\"obius 
transformations 
\begin{equation*}
\zeta\mapsto \phi_i(\zeta)=\frac{a_i\zeta+b_i}{c_i\zeta+d_i},\quad i\in [n],
\end{equation*}
as in \eqref{mobius} such that the (invertible) 
linear transformation  
$\Phi_\kappa:  \CC_\kappa[z_1,\ldots, z_n] \rightarrow 
\CC_\kappa[z_1, \ldots, z_n]$ defined by 
\begin{equation*}
\Phi_\kappa(f)(z_1,\ldots,z_n)
=(c_1z_1+d_1)^{\kappa_1}\cdots (c_nz_n+d_n)^{\kappa_n}
f(\phi_1(z_1),\ldots,\phi_n(z_n))
\end{equation*}
restricts to a bijection between $\NV_\kappa(C_1,\ldots,C_n)$ and 
$\NV_\kappa(D_1,\ldots,D_n)$.
\end{lemma}

\begin{proof}[Proof of Proposition~\ref{asymsym}]
Clearly, it is enough to prove the proposition for closed circular domains.
Suppose first that $C$ is the closed upper half-plane $\overline{\bH}_0$ 
and let 
$f \in \CC_{(1^n)}[z_1,\ldots,z_n]$ be $\overline{\bH}_0$-stable. 
Assuming, as we may, that $i=1$ and $j=2$, we need to prove that 
$$
p f(\xi_1, \xi_2, \ldots, \xi_n)+(1-p) 
f(\xi_2, \xi_1,\ldots, \xi_n) \neq 0
$$
whenever $\xi_1, \ldots, \xi_n \in \overline{\bH}_0$. By fixing 
$\xi_3,\ldots,\xi_n$ arbitrarily in $\overline{\bH}_0$ and considering 
the multi-affine polynomial in variables $z_1,z_2$ given by
$$(z_1,z_2)\mapsto g(z_1,z_2):= f(z_1,z_2,\xi_3, \ldots, \xi_n)$$ 
we see that the problem reduces to proving that for any $p\in (0,1)$ 
the polynomial 
 $p f(z,w)+(1-p)f(w, z)$ is $\overline{\bH}_0$-stable provided that $f(z,w)$ 
is $\overline{\bH}_0$-stable. This is easy to check if $d=0$ so we may assume 
that $d \neq 0$. Now, if $\{i,j\} =\{1,2\}$ then 
$$
V_i\Big( pf(z,w)+(1-p)f(w,z)\Big) 
= pV_i\Big(f(z,w)\Big)+(1-p)V_j\Big(f(z,w)\Big)
$$
which proves the proposition for $C=\overline{\bH}_0$ by 
Lemma \ref{maximum} (since $p\in (0,1)$ implies that we will be in case (4) 
of Lemma \ref{maximum}).  

Let $C$ be a closed disk and suppose that $f$ is $C$-stable. Then by 
compactness $f$ is $\tilde{C}$-stable for some open disk 
$\tilde{C} \supset C$. The result now follows by applying 
Lemma~\ref{translate} (with $\kappa=(1^n)$, $C_\ell=\tilde{C}$, $D_\ell=D$, 
$\ell\in [n]$, where $D$ is an arbitrarily fixed open half-plane) 
and using the fact the partial symmetrization operator 
commutes with the operator $\Phi_\kappa$ defined in Lemma \ref{translate}. 

The case of the closed exterior of a disk follows from the disk case 
considered above and Corollary \ref{inv} for $\kappa=(1^n)$ 
(cf.~Remark \ref{max-deg}). 
\end{proof}

In the theory of interacting particle systems \cite{Lig2} it is well known 
that the symmetrization of a polynomial $f$ can be achieved by applying 
$f \mapsto (f + \tau(f))/2$ 
infinitely many times with different transpositions $\tau$. For the sake of 
completeness, we will give a proof of this fact in the Appendix. 

\begin{proof}[Proof of Theorem~\ref{symmetrize}]
Consider first the case when $C$ is an open circular domain and recall 
Remark \ref{max-deg}. Since $\Sym(f)$ is obtained by applying 
$f \mapsto (f + \tau(f))/2$ 
infinitely many times with different transpositions $\tau$ (see 
Lemma \ref{sym-lim} below) the result follows from 
Hurwitz' Theorem \ref{mult-hur} and Proposition \ref{asymsym}. 

If $C$ is closed write 
$$
f(z,\ldots, z)= B \prod_{j=1}^d (z-c_j), 
$$
where $c_j \notin C$ for $j\in [d]$ and $B \neq 0$. Clearly, the polynomial 
$F(z_1, \ldots, z_n)$ defined 
by 
$$
F(z_1, \ldots, z_n)=B \prod_{j=1}^d (z_j-c_j)
$$
is $D$-stable, where $D$ is a suitable open circular domain containing 
$C$ but none of the $c_j$'s. Then by the above $\Sym(F)$ is $D$-stable and  
since $\Sym(F)=\Sym(f)$ the theorem follows.    
\end{proof}

\section{The Grace-Walsh-Szeg\"o 
Coincidence Theorem}\label{s-gws}

Using Theorem \ref{symmetrize}  we can give a new proof 
of the Grace-Walsh-Szeg\"o coincidence theorem that does not rely on 
apolarity theory as do most known proofs so far, see 
\cite{grace,M,RS,szego,walsh} and \S \ref{s-apol}.
We actually prove a more general version of this result which holds for 
families of circular domains.

\begin{theorem}\label{g-GWS}
Suppose $f(z_{11},\ldots,z_{1\kappa_1},\ldots,z_{n1},\ldots,z_{n\kappa_n})$ 
is a multi-affine polynomial in $|\kappa|$ complex variables 
which is symmetric in $\{z_{ij}:j\in [\kappa_i]\}$ for all $i\in [n]$, where 
$\kappa=(\kappa_1,\ldots,\kappa_n)\in\NN^n$ with $\kappa_i\ge 1$, $i\in [n]$. 
Let further $C_i$, $i\in [n]$, be circular domains and 
$\xi_{ij}\in C_i$, $j\in [\kappa_i]$, $i\in [n]$. Then there exist 
$\xi_i\in C_i$, $i\in [n]$, such that
$$
f(\xi_{11},\ldots,\xi_{1\kappa_1},\ldots,\xi_{n1},\ldots,\xi_{n\kappa_n})=
f(\xi_1,\ldots,\xi_1,\ldots,\xi_n\ldots,\xi_n)
$$
provided that $f$ has total degree $\kappa_i$ in 
$\{z_{ij} : j \in [\kappa_i]\}$ whenever $C_i$ is non-convex. 
\end{theorem}

\begin{proof}
Clearly, it is enough to prove that if the polynomial in $n$ variables 
$$
g(z_1, \ldots, z_n):= f(z_1,\ldots,z_1,\ldots,z_n\ldots,z_n)
$$ 
is $C_1\times\cdots\times C_n$-stable then 
$f(z_{11},\ldots,z_{1\kappa_1},\ldots,z_{n1},\ldots,z_{n\kappa_n})$ is 
$C_{1}^{\kappa_1}\times\cdots\times C_{n}^{\kappa_n}$-stable, which we 
will now do by considering one variable at a time. By assumption $g$ has 
degree $\kappa_i$ in the variable $z_i$ whenever $C_i$ is non-convex (symmetry 
prevents cancellation). Fix 
$\zeta_j \in C_j$, $j \in [n-1]$. The polynomial 
$
h(z_n):= g(\zeta_1, \ldots, \zeta_{n-1}, z_n)
$ 
is $C_n$-stable and we may therefore write 
$$
h(z_n)= B \prod_{j=1}^d(z_n-\alpha_j),
$$
where $B \neq 0$ and $\alpha_j \notin C_i$ for $j \in [d]$, so the polynomial 
$$
H(w_1,\ldots,  w_{\kappa_n}): =  B \prod_{j=1}^d(w_j-\alpha_j)
$$
is also $C_n$-stable. Now, if $C_n$ is non-convex then by 
Lemma \ref{maxsupport} one has $d=\kappa_n$ and by Theorem \ref{symmetrize} 
the symmetrization operator $\Sym$ acting on $\kappa_n$ variables maps 
$\bH_0$ to a $C_n$-stable polynomial. Since the numbers 
$\zeta_i$, $i \in [n-1]$,
were arbitrary this means that the polarization of $g$ that splits the 
variable $z_n$ symmetrically into $\kappa_n$ new variables, i.e., the linear 
operator 
$$
g(z_1,\ldots, z_n) \mapsto f(z_1,\ldots, z_1,\ldots, z_{n-1},\ldots, z_{n-1}, 
z_{n1}, \ldots, z_{n\kappa_n})
$$
preserves the stability in question. By polarizing one variable at a time we 
conclude that 
$f(z_{11},\ldots,z_{1\kappa_1},\ldots,z_{n1},\ldots,z_{n\kappa_n})$ is 
$C_{1}^{\kappa_1}\times\cdots\times C_{n}^{\kappa_n}$-stable.
\end{proof}

\begin{remark}\label{r-ruelle}
In \cite{ruelle1} Ruelle produced a proof of the Grace-Walsh-Szeg\"o 
coincidence theorem (Theorem \ref{GWS}) using similar ideas. 
\end{remark}

\begin{remark}\label{walsh-w}
A yet more 
general version of Theorem \ref{g-GWS} was actually given by Walsh in 
\cite[Theorem 1]{walsh} without assuming any degree conditions, the only 
requirement in \cite{walsh} being that $C_i$, $i\in [n]$, 
should be just (closed) circular domains. However, in such generality 
Walsh's aforementioned result fails already for $n=1$.
\end{remark}

%\begin{remark}\label{r-horm}
%Polarization procedures like the one in the proof of Theorem \ref{g-GWS}  
%are often used in the 
%theory of hyperbolic polynomials and partial differential equations 
%\cite{ABG,BGLS,garding,hor2,Lax,Re}. They were also useful for 
%proving the sufficiency part of the main results in \cite{BB-I}. More 
%precisely, if $\kappa =(\kappa_1,\ldots,\kappa_n)\in \NN^n$ and $\CM^\kappa$ 
%is the space of multi-affine polynomials in the variables 
%$\{ z_{ij} : 1\leq i \leq n, 1\leq j \leq \kappa_i\}$  
%define a (linear) polarization operator  
%$
%\Pi_\kappa^\uparrow : \CC_\kappa[z_1, \ldots, z_n] \rightarrow \CM^\kappa
%$ 
%that associates to each $f\in \CC_\kappa[z_1, \ldots, z_n]$ 
%the unique polynomial 
%$\Pi_\kappa^\uparrow(f) \in \CM^\kappa$ such that 
%\begin{itemize}
%\item[(a)] for any $1 \leq i \leq n$ the polynomial  
%$\Pi_\kappa^\uparrow(f)$ is symmetric in 
%$\{ z_{ij} : 1\leq j \leq \kappa_i\}$; 
%\item[(b)] if we let  $z_{ij}=z_i$ for all $1\le i\le n$ and 
%$1\le j\le \kappa_i$
%in $\Pi_\kappa^\uparrow(f)$ we recover $f$. 
%\end{itemize}

%Note that Theorem~\ref{g-GWS} applies to {\em all}  
%multi-affine polynomials obtained via the polarization operator 
%$\Pi_\kappa^\uparrow$. It would be interesting to study the connections 
%between Theorem \ref{g-GWS} and H\"ormander's Laguerre type separation 
%and ratio theorems (for homogeneous polynomials) \cite{hor1} or their various 
%extensions \cite{JZ,mar}.
%\end{remark}

\section{Master Composition Theorems}\label{s-hilf}

Composition (or convolution) theorems such as 
the Schur-Mal\'o-Szeg\"o theorems (\cite[Theorem~2.4]{CC1}, 
\cite[Theorem~3.4.1d]{RS}), the 
Cohn-Egerv\'ary-Szeg\"o theorem (\cite[Theorem~3.4.1d]{RS}), Walsh's
theorems (\cite[Theorems~3.4.2c and 5.3.1]{RS}) or de Bruijn's theorems 
\cite{db1,db2} play an important role in the 
analytic theory of univariate complex polynomials and allow to locate their 
zeros in certain circular domains \cite{M,RS}.

Using results of \cite{BB-I} we establish ``master composition 
theorems'' that provide a unifying framework
for {\em multivariate} generalizations of the classical 
theorems mentioned above. Let us first recall two of the classification 
theorems from \cite{BB-I}. 

\begin{theorem}\label{open-halfplane}
Let $\kappa\in\NN^n$, 
$T : \CC_\kappa[z_1,\ldots,z_n] \rightarrow \CC[z_1,\ldots,z_n]$ 
be a linear operator, and $C=\bH_\te$ for some $0\le \te<2\pi$. 
Then $T$ preserves $C$-stability if and only if 
\begin{itemize}
\item[(a)] $T$ has range of dimension at most one and is of the form 
$T(f) = \alpha(f)P$, 
where $\alpha$ is a linear functional on $\CC_\kappa[z_1,\ldots, z_n]$ and 
$P$ is a $C$-stable polynomial, or
\item[(b)] The polynomial (in $2n$ variables) 
\begin{equation}\label{stablesymb}
T[(z+w)^\kappa]:=\sum_{\alpha \leq \kappa} \binom \kappa \alpha 
T(z^\alpha)w^{\kappa-\alpha}
\end{equation}
is $C$-stable. 
\end{itemize}
\end{theorem}

\begin{remark}\label{gl}
Theorem \ref{open-halfplane} trivially implies the following 
(well-known) multivariate Gauss-Lucas theorem: if 
$f\in \CC_\kappa[z_1,\ldots,z_n]$ is $\bH_\te$-stable for some $0\le \te<2\pi$
then $\partial f/ \partial z_i$ is $\bH_\te$-stable or identically zero 
for any $i\in [n]$. More generally, the $(n+1)$-variable polynomial  
$f+z_{n+1}\partial f/ \partial z_i$ is $\bH_\te$-stable. 
\end{remark}

\begin{theorem}\label{open-disk}
Let $\kappa\in\NN^n$, 
$T : \CC_\kappa[z_1,\ldots,z_n] \rightarrow \CC[z_1,\ldots,z_n]$ be a linear 
operator, and $C=\DD$ or $\bH_{\frac{\pi}{2}}$. 
Then $T$ preserves $C$-stability if and only if 
\begin{itemize}
\item[(a)] $T$ has range of dimension at most one and is of the form 
$T(f) = \alpha(f)P$, 
where $\alpha$ is a linear functional on $\CC_\kappa[z_1,\ldots, z_n]$ and 
$P$ is a $C$-stable polynomial, or
\item[(b)] The polynomial (in $2n$ variables) 
\begin{equation}\label{stablesymb-d}
T[(1+zw)^\kappa]:=\sum_{\alpha \leq \kappa} \binom \kappa \alpha 
T(z^\alpha)w^{\alpha}
\end{equation}
is $C$-stable. 
\end{itemize}
\end{theorem}

The polynomials in \eqref{stablesymb} and \eqref{stablesymb-d} are called the 
{\em algebraic symbols} of $T$ with respect to the circular domains under 
consideration (for $\bH_{\frac{\pi}{2}}$ 
it is often more convenient -- but equivalent -- to choose 
\eqref{stablesymb-d} rather than \eqref{stablesymb}, 
cf.~\cite[Remark 6.1]{BB-I}). 

The main result of this section is as follows.

\begin{theorem}\label{uvzw}
Let $\kappa=(\kappa_1,\ldots,\kappa_n)\in\NN^n$, 
$f(u,v) \in \CC[u_1,\ldots,u_n,v_1,\ldots,v_n]$, and 
$g(z,w) \in \CC[z_1,\ldots, z_n, w_1,\ldots, w_n]$. Suppose that 
$\deg_{u_i}(f)\le \kappa_i$ and $\deg_{z_i}(g)\le \kappa_i$ for all 
$i \in [n]$.
\begin{itemize}
\item[(a)] If $f$ and $g$ are $\bH_\te$-stable for some 
$0\le \te<2\pi$, then the polynomial (in $4n$ variables)
$$
 \sum_{\alpha \leq \kappa}
 \frac {\partial^\alpha f} {\partial u^\alpha}(u,v)\cdot 
\frac {\partial^{\kappa-\alpha}g}{\partial z^{\kappa-\alpha}}(z,w)
$$
is $\bH_\te$-stable or identically zero. 
\item[(b)] If $f$ and $g$ are $\bH_0$-stable, then the polynomial 
$$
 \sum_{\alpha \leq \kappa}
(-1)^\alpha \frac{(\kappa-\alpha)!}{\alpha!}\cdot
\frac {\partial^\alpha f} {\partial u^\alpha}(u,v)\cdot 
\frac {\partial^{\alpha}g}{\partial z^{\alpha}}(z,w)
$$
is $\bH_0$-stable or identically zero. 
\item[(c)] If $f$ and $g$ are $\DD$-stable, then the polynomial 
$$
 \sum_{\alpha \leq \kappa}
\frac{(\kappa-\alpha)!}{\alpha!}\cdot 
\frac {\partial^\alpha f} {\partial u^\alpha}(0,v)\cdot 
\frac {\partial^{\alpha}g}{\partial z^{\alpha}}(z,w)
$$
is $\DD$-stable or identically zero. 
\end{itemize}
\end{theorem}

\begin{proof}
We only prove (a) since the proofs of (b) and (c) are almost identical. 
Let $\gamma=(\gamma_1,\ldots,\gamma_n) \in \NN^n$ be fixed. Define a  
$\CC[u_1,\ldots, w_n]$-valued linear operator $T$ on the space of all 
polynomials $h$ in $4n$ variables $u_1,\ldots, w_n$ satisfying 
$\deg_{u_j}(h) \leq \kappa_j$, $\deg_{v_j}(h) \leq \gamma_j$, 
$\deg_{z_j}(h) \leq \gamma_j$, and $\deg_{z_j}(h) \leq \gamma_j$ for all 
$j \in [n]$ by setting 
$$
T[h(u,v,z,w)]=  \sum_{\alpha \leq \kappa}
 \frac {\partial^\alpha h} {\partial u^\alpha}(u,v,z,w)\cdot 
\frac {\partial^{\kappa-\alpha}g}{\partial z^{\kappa-\alpha}}(z,w). 
$$
Let $\tilde{u}=(\tilde{u}_1, \ldots, \tilde{u}_n)$, 
$\tilde{v}=(\tilde{v}_1, \ldots, \tilde{v}_n)$, 
$\tilde{z}=(\tilde{z}_1, \ldots, \tilde{z}_n)$ and 
$\tilde{w}=(\tilde{w}_1, \ldots, \tilde{w}_n)$ be new sets of variables. 
The symbol of $T$ with respect to $\bH_\te$, that is,  
\begin{equation*}
\begin{split}
T[(u&+\tilde{u})^\kappa(v+\tilde{v})^\gamma(z+\tilde{z})^\gamma
(w+\tilde{w})^\gamma]\\ 
&=(v+\tilde{v})^\gamma(z+\tilde{z})^\gamma(w+\tilde{w})^\gamma 
\sum_{\alpha \leq \kappa}
 \frac {\kappa!} {(\kappa-\alpha)!}(u+\tilde{u})^{\kappa-\alpha}
\frac {\partial^{\kappa-\alpha}g}{\partial z^{\kappa-\alpha}}(z,w)\\
&=\kappa! (v+\tilde{v})^\gamma(z+\tilde{z})^\gamma(w+\tilde{w})^\gamma 
g(z+u+\tilde{u}, w)
\end{split}
\end{equation*}
is clearly $\bH_\te$-stable which proves (a) by Theorem \ref{open-halfplane} 
since $\gamma\in \NN^n$ was arbitrary. 
\end{proof}

An important special case of the above theorem is particularly 
attractive, as is its proof. 

\begin{corollary}\label{master-comp}
Let $\kappa \in \NN^n$ and 
$f,g \in \CC[z_1,\ldots, z_n,w_1,\ldots,w_n]$ be of the form 
$$
f(z,w)=\sum_{\alpha \leq \kappa} \binom \kappa \alpha P_\alpha(w)z^\alpha, 
\quad 
g(z,w)=\sum_{\alpha \leq \kappa} \binom \kappa \alpha Q_\alpha(z)w^\alpha,
$$
where $z=(z_1,\ldots,z_n)$, $w=(w_1,\ldots,w_n)$. 
\begin{itemize}
\item[(a)] If $f$ and $g$ are $\bH_\te$-stable for some 
$0\le \te<2\pi$, then so is the polynomial
$$
\sum_{\alpha \leq \kappa} \binom \kappa \alpha P_\alpha(w)Q_{\kappa-\alpha}(z)=
\frac 1 {\kappa!} \sum_{\alpha \leq \kappa}
 \frac {\partial^\alpha f} {\partial z^\alpha}(0,w)\cdot 
\frac {\partial^{\kappa-\alpha}g}{\partial w^{\kappa-\alpha}}(z,0) 
$$
unless it is identically zero.
\item[(b)] If $f$ and $g$ are $\bH_0$-stable, then so is the polynomial
$$
\sum_{\alpha \leq \kappa} (-1)^{\alpha}\binom \kappa \alpha 
P_\alpha(w)Q_{\alpha}(z)=
\frac 1 {\kappa!} \sum_{\alpha \leq \kappa}(-1)^{\alpha}
\frac{(\kappa-\alpha)!}{\alpha!}\cdot
 \frac {\partial^\alpha f} {\partial z^\alpha}(0,w)\cdot 
\frac {\partial^{\alpha}g}{\partial w^{\alpha}}(z,0) 
$$
unless it is identically zero. 
\item[(c)] If $f$ and $g$ are $\DD$-stable, then so is the polynomial
$$
\sum_{\alpha \leq \kappa} \binom \kappa \alpha P_\alpha(w)Q_{\alpha}(z)
=\frac 1 {\kappa!} \sum_{\alpha \leq \kappa}
\frac{(\kappa-\alpha)!}{\alpha!}\cdot
 \frac {\partial^\alpha f} {\partial z^\alpha}(0,w)\cdot 
\frac {\partial^{\alpha}g}{\partial w^{\alpha}}(z,0) 
$$
unless it is identically zero.
\end{itemize}
\end{corollary}

\begin{proof} 
Suppose that $f,g$ are as in part (a) of the corollary. Let 
$$
T: \CC_\beta[z_1,\ldots, z_n] \rightarrow \CC_\kappa[z_1,\ldots,z_n]
\,\text{ and }\,S: \CC_\kappa[z_1,\ldots, z_n] \rightarrow 
\CC_\gamma[z_1,\ldots,z_n]
$$ 
be the linear operators whose algebraic symbols with respect to $\bH_\te$ 
(cf.~\eqref{stablesymb}) are $f$, respectively $g$, with 
$\beta,\gamma\in \NN^n$ appropriately chosen. 
By Theorem \ref{open-halfplane} both $S$ and $T$ preserve 
$\bH_\te$-stability, hence so does their (operator) composition 
$ST$ whose symbol is precisely the polynomial in (a). Applying
Theorem \ref{open-halfplane} again we conclude that this polynomial is 
$\bH_\te$-stable unless it is of the form $A(z)B(w)$ for some polynomials 
$A$ and $B$. If this is the case and these polynomials are not 
identically zero then $A(z)$ must be 
$\bH_\te$-stable (being the polynomial $P$ in part (a) of 
Theorem \ref{open-halfplane}) and by exchanging the roles of $f$ and $g$ 
we get that $B(w)$, thus also $A(z)B(w)$, must be $\bH_\te$-stable. 
This proves (a). Parts (b) and (c) follow similarly.
\end{proof}

\begin{example}\label{ex1}
Let us show how the classical (univariate) 
Schur-Mal\'o-Szeg\"o theorem can be easily derived from Corollary 
\ref{master-comp}. If $\sum_{k=0}^n \binom n k a_k z^k$ and 
$\sum_{k=0}^n \binom n k b_k z^k$ are two polynomials with real zeros only 
and in addition all the zeros of the latter polynomial are non-positive then 
the bivariate polynomials 
\begin{equation*}
f(z,w)=\sum_{k=0}^n \binom n k a_kz^k\,\text{ and }\,
g(z,w)=\sum_{k=0}^n  \binom n kb_{n-k}z^{n-k} w^k
\end{equation*}
are stable (note that $f$ actually depends only on $z$). 
Corollary \ref{master-comp} (a) implies that the Schur-Mal\'o-Szeg\"o 
composition of the two given polynomials, i.e., the univariate polynomial 
$$
\sum_{k=0}^n \binom n k a_kb_kz^k
$$
is also stable (that is, real-rooted). 
\end{example}

\begin{example}\label{ex2}
Theorems 3.11 and 4.6 in \cite{BBS3} (the former actually
follows from the latter) provide some multivariate
extensions of the classical composition results mentioned above. 
In particular, \cite[Theorem 4.6]{BBS3} shows
that the Weyl product of polynomials
(defined via the product formula in the Weyl algebra) preserves stability, 
that is, if $f(z,w)$ and $g(z,w)$ are stable polynomials in 
$z=(z_1, \ldots, z_n)$ and $w=(w_1, \ldots, w_n)$ then so is the polynomial 
$$
\sum_{\alpha \in \NN^n}\frac {(-1)^\alpha} {\alpha !} 
\frac{\partial^\alpha f}{\partial z^\alpha}(z,w)\cdot 
\frac{\partial^\alpha g}{\partial w^\alpha}(z,w).
$$
To see that this is in fact a consequence of Theorem \ref{uvzw} (b) 
let $\kappa_N=(N,\ldots,N)\in\NN^n$ and let $f$ and $g$ be stable 
polynomials as in the statement of Theorem \ref{uvzw}. Then since stability 
is closed under scaling the variables with positive numbers the polynomial 
$$
H_N(u,v,z,w)= \sum_{\alpha \leq \kappa}
(-1)^\alpha \frac{(\kappa-\alpha)!}{\alpha!}\cdot 
\frac {\kappa_N^\alpha (\kappa_N-\alpha)!}{\kappa_N!}\cdot 
\frac {\partial^\alpha f} {\partial u^\alpha}(Nu,v)\cdot 
\frac {\partial^{\alpha}g}{\partial z^{\alpha}}(z,w)
$$
is stable for large $N$. But then the polynomial 
$$
\lim_{N \rightarrow \infty} H_N(u/N,v,z,w)=
\sum_{\alpha}
 \frac{(-1)^\alpha}{\alpha!}\frac {\partial^\alpha f} {\partial u^\alpha}(u,v)
\cdot 
\frac {\partial^{\alpha}g}{\partial z^{\alpha}}(z,w)
$$
is stable or identically zero, as claimed. 

We conclude with one further consequence.  
For $t=(t_1,\ldots,t_n)\in \RR_+^n$ define the $t$-{\em deformed  
Weyl product} of $f$ and $g$ by 
$$
\sum_{\alpha \in \NN^n}\frac {(-1)^\alpha t^{\alpha}} {\alpha !} 
\frac{\partial^\alpha f}{\partial u^\alpha}(u,v)\cdot 
\frac{\partial^\alpha g}{\partial w^\alpha}(z,w).
$$
Using again the fact that stability is closed under scaling the variables 
with positive constants we deduce from above that the $t$-deformed Weyl 
product preserves stability. As a 
special (univariate) case, note that if $n=1$, $t<0$ and 
$f\in\RR[z]\setminus\{0\}$, $g\in\RR[w]\setminus\{0\}$ have all real zeros 
then by the above 
the univariate polynomial
$$
\sum_{k \in \NN}\frac {t^{k}} {k!} f^{(k)}(z)\cdot g^{(k)}(w)\big|_{w=z}=
\sum_{k \in \NN}\frac {t^{k}} {k!} f^{(k)}(z)\cdot g^{(k)}(z)
$$
has all real zeros. We thus recover de Bruijn's 
\cite[Theorem 2]{db1} and \cite[Lemma 1]{db2}.
\end{example}

\section{Hard P\'olya-Schur Theory: Bounded Degree Multiplier 
Sequences}\label{s-hard-ps}

Using the characterization of linear operators preserving real stability 
obtained in \cite{BB-I}  
we can establish ``hard'' (bounded 
degree) multivariate versions of P\'olya-Schur's classification of 
{\em multiplier sequences} \cite{PS} that extend the ``soft''
(unbounded degree) theorems of \cite{BBS3}. 

A sequence of real numbers 
$\{\lambda(k)\}_{k\in\NN}$ is called a multiplier sequence if 
the linear operator 
on univariate polynomials defined by $T(z^k)=\lambda(k)z^k$, $k\in\NN$, 
preserves real-rootedness, that is, $T(f)\in\HH_1(\RR)\cup \{0\}$
whenever $f\in\HH_1(\RR)$.  
A {\em multivariate multiplier sequence} is then defined as a sequence 
$\{\lambda(\alpha)\}_{\alpha \in \NN^n}$ of real numbers 
such that the linear operator 
$T: \RR[z_1,\ldots, z_n] \rightarrow \RR[z_1, \ldots, z_n]$ defined by 
$T(z^\alpha)= \lambda(\alpha)z^\alpha$, $\alpha \in \NN^n$, 
preserves real stability, see \cite{BBS3}. These were characterized in 
\cite{BBS3} but here we will prove the corresponding ``hard'' theorems. 
Given $\kappa \in \NN^n$ we say that a 
sequence $\{\lambda(\alpha)\}_{\alpha \leq \kappa}$ of real numbers is a 
$\kappa$-{\em multiplier sequence} if  the linear operator 
$T: \RR_\kappa[z_1,\ldots, z_n] \rightarrow \RR_\kappa[z_1, \ldots, z_n]$  
defined by $T(z^\alpha)= \lambda(\alpha)z^\alpha$, $\alpha \leq \kappa$, 
preserves real stability. This is the multivariate generalization of 
$n$-multiplier sequences \cite{CC2}. 

Recall the following lemma from \cite{BB-I}. 

\begin{lemma}\label{krein}
Let $f,g \in \RR[z_1,\ldots, z_n]\setminus \{0\}$, set $h = f+ig$ and 
suppose that $f$ and $g$ are not constant multiples of each other. 
The following  are equivalent: 
\begin{enumerate}
\item $h$ is stable;
\item $|h(z)|>|h(\bar{z})|$ for all $z=(z_1,\ldots,z_n)\in\CC^n$ with 
$\Im(z_j)>0$, $j\in [n]$, where $\bar{z}=(\bar{z}_1,\ldots,\bar{z}_n)$; 
\item $f+z_{n+1}g\in\HH_{n+1}(\CC)$;
\item $f$ and $g$ are stable and 
$$
\text{{\em Im}}\!\left(\frac{f(z)}{g(z)}\right)\ge 0
$$
whenever $z=(z_1,\ldots,z_n)\in\CC^n$ with $\Im(z_j)>0$, $j\in [n]$.
\end{enumerate}
\end{lemma}

A polynomial 
$f\in\CC[z_1,\ldots,z_n]\setminus \{0\}$ is said to have the
{\em same-phase property} if there exists $\alpha\in\RR$ such that
all the non-zero coefficients of $e^{-i\alpha}f$ are positive. The next 
lemma was first proved in \cite[Theorem 6.1]{COSW}. We will use it in the
proof of Lemma \ref{multi-symbol} below and for completeness
we provide here a self-contained proof based on our results so far. 

\begin{lemma}\label{phase}
If $f\in\CC[z_1,\ldots,z_n]\setminus \{0\}$ is homogeneous of degree $d$ 
and $\bH_\theta$-stable for some $\theta\in [0,2\pi)$ then $f$ has
the same-phase property.
\end{lemma}

\begin{proof}
Note first that the assumptions of the lemma imply that $f$
is $\bH_{\theta'}$-stable for any $\theta'\in[0,2\pi)$. Without loss of 
generality we may also assume that $\partial_jf\not\equiv 0$, where 
$\partial_j=\frac{\partial}{\partial z_j}$. 
We will now use induction on $d$.  
The statement is trivially true for $d=0$ so suppose $d\ge 1$. 
Applying $\frac{\partial}{\partial t}$ to the 
identity $f(tz_1,\ldots,tz_n)=t^df(z_1,\ldots,z_n)$ and setting
$t=1$ we get
\begin{equation}\label{n-sum}
f(z_1,\ldots,z_n)=d^{-1}\sum_{j=1}^{n}z_j\partial_jf(z_1,\ldots,z_n). 
\end{equation}
Each polynomial $\partial_jf$ is stable (e.g.~by Theorem \ref{open-halfplane}) 
and homogeneous of degree $d-1$. By the induction hypothesis 
there exists $\alpha_j\in\RR$ such that $e^{-i\alpha_j}\partial_jf$ has all
non-negative coefficients. In view of \eqref{n-sum} it is therefore 
enough to show 
that $\alpha_j\equiv \alpha_k \bmod 2\pi$, $j,k\in [n]$. For each $j\in [n]$ 
we get by Remark \ref{gl}, Lemma \ref{krein} 
(3) $\Leftrightarrow$ (4) and homogeneity that 
\begin{equation}\label{frac1}
\Im\!\left(\frac{\partial_jf(z)}{f(z)}\right)\le 0, 
\quad z=(z_1,\ldots,z_n)\in \bH_0^n,
\end{equation}
and
\begin{equation}\label{frac2}
\Re\!\left(\frac{\partial_jf(z)}{f(z)}\right)\ge 0, 
\quad z=(z_1,\ldots,z_n)\in \bH_{\frac{\pi}{2}}^n.
\end{equation}
By continuity \eqref{frac1} also holds for all $z \in \RR^n$  
for which $f(z) \neq 0$. Using homogeneity we see that 
$\Im(\partial_j f(-z)/f(-z))=-\Im(\partial_j f(z)/f(z))$  
hence $\partial_j f(z)/f(z)$ is a real rational function. Since 
$e^{-i\alpha_j}\partial_jf(\zeta)\in\RR$ we deduce that 
$e^{-i\alpha_j}f(\zeta)\in\RR$ for $j\in [n]$ and $\zeta\in\RR^n$. 
Now, from \eqref{frac2}
with $z\in\RR_+^n$ and the fact that $e^{-i\alpha_j}\partial_jf(z)>0$ for all
such $z$ we conclude that $e^{-i\alpha_j}f(z)>0$ whenever $z\in\RR_+^n$, 
$j\in [n]$, and thus  
$\alpha_j\equiv \alpha_k \bmod 2\pi$, $j,k\in [n]$, as required. 
\end{proof}

\begin{lemma}\label{multi-symbol}
Let $f(z,w)= \sum_{\alpha \in \NN^n}a(\alpha)z^\alpha w^\alpha 
\in \CC[z_1,\ldots, z_n, w_1,\ldots, w_n]$. Then $f$ is stable if and only if 
it can be written as 
$$
f(z,w)= Cf_1(z_1w_1)\cdots f_n(z_nw_n), 
$$
where $C \in \CC$ and $f_1(t), \ldots, f_n(t)$ are univariate real 
polynomials with real and non-negative zeros only.
\end{lemma}

\begin{proof}
The sufficiency part follows simply by noticing that if $\mu \leq 0$ then 
$\mu +zw$ is a stable polynomial in two variables. 

Suppose that $f(z,w)= \sum_{\alpha \in \NN^n}a(\alpha)z^\alpha w^\alpha$ 
is stable. We claim that its support 
$J:=\supp(f)$ 
has unique minimal and maximal elements with respect to the standard partial 
order on $\NN^n$. Assume the contrary, let 
 $\alpha, \alpha'$ be two different minimal elements and let $i, j$ be 
indices such that 
 $\alpha_i>\alpha'_j$ and $\alpha_j<\alpha'_i$.  
By \cite[Theorem~3.2]{Br1} $J$ is a jump system and since 
$\alpha''=\alpha-e_i \notin J$, there is an index $k \neq i$ such 
that $\alpha''+e_k \in J$. Let $g(z_i,w_i,z_k,w_k)$ be the  polynomial 
$$
\frac {\partial^{\alpha''}}{\partial z^{\alpha''}}
\frac {\partial^{\alpha''}}{\partial w^{\alpha''}} f\Big|_{z_\ell=w_\ell=0,\, 
\ell \notin \{i,k\}} 
$$
and set 
$$
h(z_i,w_i,z_k,w_k) = \lim_{\lambda \rightarrow 0^+} 
\lambda^{-2}g(\lambda z_i,\lambda w_i,\lambda z_k,\lambda w_k). 
$$
By Hurwitz' theorem (Theorem \ref{mult-hur}) $h$ is a stable polynomial. 
However, by construction $h$ is of the form $Az_iw_i + Bz_kw_k$ with 
$AB\neq 0$, which is a contradiction since polynomials of this type cannot 
be stable. This shows that $J$ has a unique minimal element.

If $f(z,1)$ has degree at most $\kappa_i$ in the variable $z_i$, 
$i\in [n]$, we may consider the stable 
polynomial 
$z^\kappa w^\kappa f(-z^{-1}, -w^{-1})$, where 
$\kappa=(\kappa_1,\ldots,\kappa_n)$, $z^{-1}=(z_1^{-1},\ldots,z_n^{-1})$ 
and similarly for $w^{-1}$. By the above the support of the 
latter polynomial has a unique 
minimal element, thus providing a unique maximal element for the support 
$J$ of $f$. 

Let now $\xi, \kappa$ be the minimal, respectively maximal element of $J$ 
and let 
$T : \CC_\kappa[z_1,\ldots, z_n] \rightarrow \CC_\kappa[z_1,\ldots, z_n]$ be 
the linear operator defined by 
$$
T[z^\alpha]= \lambda(\alpha)z^\alpha=(-1)^\alpha \binom{\kappa}{\alpha}^{-1}
a(\alpha)z^\alpha, \quad 0\leq \alpha \leq \kappa. 
$$
Let $\{e_j\}_{j=1}^n$ be the standard orthonormal basis of $\RR^n$. 
We want to show that 
\begin{equation}\label{split}
\lambda(\alpha)= \lambda(\xi)^{-n+1}\lambda(\xi +(\alpha_1-\xi_1)e_1)\cdots 
\lambda(\xi +(\alpha_n-\xi_n)e_n), \quad \xi \leq \alpha \leq \kappa. 
\end{equation}
This will then prove the necessity part since $f$ will split into a product 
as in the statement of the lemma and the polynomials 
$f_j(t)$, $1\leq j \leq n$, will have the desired properties since 
$f_j(z_jw_j)$ is necessarily stable. 

To prove \eqref{split} note that the algebraic symbol of $T$ is given by 
$$
G_T(z,w)
=\sum_{\alpha\le \kappa} (-1)^\alpha a(\alpha)z^\alpha w^{\kappa -\alpha}
=w^\kappa f(z,-w^{-1}),
$$
which is stable. By Theorem \ref{open-halfplane} $T$ preserves 
stability. 
Since $G_T(z,w)$ is homogeneous we may assume  
that $(-1)^\alpha a(\alpha) \geq 0$ for all $\alpha$ in view of 
Lemma~\ref{phase}. Now, it is easy to check that 
\begin{equation}\label{peacy}
a +bz +cw+dzw \in \RR[z,w] \mbox{ is stable } \Longleftrightarrow bc \geq ad  
\end{equation}
(see, e.g., \cite{Br1} or
just adapt the arguments in the proof of Lemma \ref{maximum})  
and of course 
\begin{equation}\label{easypeacy}
a +2bz +cz^2 \in \RR[z] \mbox{ is stable } \Longleftrightarrow b^2 \geq ac. 
\end{equation} 
Let $\gamma \in \NN^n$ and $1\leq i,j\leq n$ be such that 
$\gamma +e_i +e_j \leq \kappa$. Applying 
$T$ to the polynomials $z^\gamma (1+z_i)(1+z_j)$ and 
$z^\gamma (1-z_i)(1+z_j)$ and keeping  
\eqref{peacy} and \eqref{easypeacy} in mind we see that 
$\lambda(\gamma)\lambda(\gamma+e_i+e_j)\geq 
\lambda(\gamma+e_i)\lambda(\gamma+e_j)$ and 
$\lambda(\gamma)\lambda(\gamma+e_i+e_j)\leq 
\lambda(\gamma+e_i)\lambda(\gamma+e_j)$ hence
\begin{equation}\label{lambdas}
\lambda(\gamma)\lambda(\gamma+e_i+e_j) 
= \lambda(\gamma+e_i)\lambda(\gamma+e_j)\, 
\mbox{ whenever } \gamma \in \NN^n \mbox{ and } \gamma +e_i +e_j \leq \kappa.
\end{equation}
From \eqref{lambdas} and \cite[Corollary 3.7]{Br1} we deduce that 
$\lambda(\gamma)>0$ for all $\xi \leq \gamma \leq \kappa$.   
The proposed formula \eqref{split} now follows by induction over 
$k:=|\alpha|-|\xi|$. 
\end{proof}

Recall the following theorem from \cite{BB-I}.

\begin{theorem}\label{multi-finite-hyp}
Let $\kappa \in \NN^n$ and 
$T : \RR_\kappa[z_1,\ldots, z_n] \rightarrow \RR[z_1,\ldots, z_n]$ be a 
linear operator. Then $T$ preserves real stability if and only if either
\begin{itemize}
\item[(a)] $T$ has at most $2$-dimensional range and is given by 
$
T(f) = \alpha(f)P + \beta(f)Q,
$
where $\alpha, \beta$ are real linear forms on $\RR_\kappa[z_1,\ldots, z_n]$
and $P,Q\in\HH_n(\RR)$ are such that $P +iQ\in\HH_n(\CC)$, or 
\item[(b)] Either $T[(z+w)^\kappa]\in \HH_{2n}(\RR)$ or 
$T[(z-w)^\kappa]\in \HH_{2n}(\RR)$.
\end{itemize}
\end{theorem}

The ``hard'' multivariate version of P\'olya-Schur's theorem \cite{PS} is 
as follows.

\begin{theorem}\label{kappa-mult}
Let $\kappa \in \NN^n$, let 
$\lambda:=\{\lambda(\alpha)\}_{\alpha \leq \kappa}$ be a 
sequence of real numbers and 
$T: \RR_\kappa[z_1,\ldots, z_n] \rightarrow \RR_\kappa[z_1, \ldots, z_n]$ 
 be the corresponding linear operator. The following are equivalent:
\begin{itemize}
\item[(a)] $\lambda$ is a $\kappa$-multiplier sequence; 
\item[(b)] $\pm \lambda$ is the product of one-dimensional 
$\kappa_i$-multiplier sequences that are either all alternating in sign or 
all non-negative, i.e., 
$$
\pm \lambda(\alpha) = \lambda_1 (\alpha_1) \cdots \lambda_n(\alpha_n), 
\quad \alpha \leq \kappa,
$$
where $\lambda_i$ is a $\kappa_i$-multiplier sequence, $i\in [n]$,  
and either all 
$\lambda_i$'s are non-negative or all are alternating in sign; 
\item[(c)] Either $T[(z+w)^\kappa]\in \HH_{2n}(\RR)$ or  
$T[(z-w)^\kappa]\in \HH_{2n}(\RR)$; 
\item[(d)] $T[(z+w)^\kappa]$ or  $T[(z-w)^\kappa]$ can be written as 
$$
f_1(z_1w_1)\cdots f_n(z_nw_n), 
$$
where $f_1(t), \ldots, f_n(t)$ are univariate polynomials with real zeros 
only, and all these zeros have the same sign (collectively).
\end{itemize}
\end{theorem}

\begin{proof}
This is an immediate consequence of Theorem \ref{multi-finite-hyp} and
 Lemma \ref{multi-symbol}. 
\end{proof}

\begin{remark}\label{k-stab}
Note that the $\kappa$-multiplier sequences with constant sign are precisely 
the sequences whose corresponding operators preserve stability. 
\end{remark}

\begin{corollary}\label{multi-Hurwitz}
Let $\kappa \in \NN^n$,  
$\lambda:=\{\lambda(\alpha)\}_{\alpha \leq \kappa}$ be a sequence of complex 
numbers and 
$T: \CC_\kappa[z_1,\ldots, z_n] \rightarrow \CC_\kappa[z_1, \ldots, z_n]$  
be the corresponding linear operator. Then $T$ preserves weak 
Hurwitz stability 
if and only if $\lambda$ is (a constant complex multiple of) a non-negative 
$\kappa$-multiplier sequence. 
\end{corollary}

\begin{proof}
By Theorem \ref{open-halfplane}  $T$ preserves 
weak Hurwitz stability if and only if the polynomial 
$$
T[(z+w)^\kappa]= \sum_{\alpha \leq \kappa} \binom \kappa \alpha 
\lambda(\alpha)z^\alpha w^{\kappa-\alpha} 
$$
is weakly Hurwitz stable, which occurs exactly 
when 
$$
\sum_{\alpha \leq \kappa} (-1)^\alpha\binom \kappa \alpha 
\lambda(\alpha)z^\alpha w^\alpha 
$$
is stable. The assertion now follows from Lemma~\ref{multi-symbol} and 
Remark \ref{k-stab}.
\end{proof}

\section{Multivariate Apolarity}\label{s-apol}

The goal of this section is to develop a higher-dimensional apolarity theory 
and establish Grace type theorems for arbitrary multivariate 
polynomials.
%Let us first define a notion that extends the one given in 
%\cite{hink} for multi-affine polynomials (cf.~\cite{grace,M,RS,szego} in the 
%univariate case).

Two univariate polynomials 
$f(z)= \sum_{k=0}^n \binom n k a_k z^k$ and 
$g(z)= \sum_{k=0}^n \binom n k b_k z^k$ of degree at most $n$ are 
{\em apolar} if 
$$
\{f,g\}_n:= \sum_{k=0}^n (-1)^k f^{(k)}(0)g^{(n-k)}(0) 
= \frac 1 {n!}\sum_{k=0}^n (-1)^k\binom n k a_kb_{n-k} =0.
$$
Grace's classical apolarity theorem is as follows \cite{grace,M,RS,szego}. 

\begin{theorem}[Grace] 
Let $f$ and $g$ be apolar polynomials of degree $n\ge 1$. 
If $f$ has all zeros in a circular domain $C$ then $g$ has at least one zero 
in $C$. 
\end{theorem}

Note that we may reformulate Grace's theorem as follows. 

\begin{theorem} 
Let $f$ and $g$ be polynomials of degree $n\ge 1$ and $C$ be a 
circular domain. 
If $f$ is $C$-stable and $g$ is $\CC \setminus C$-stable then 
$\{f,g\}_n \neq 0$. 
\end{theorem}

For two polynomials $f,g \in \CC[z_1,\ldots, z_n]$ and $\kappa \in \NN^n$ 
define 
$$
\{f,g\}_\kappa := 
\sum_{\alpha \leq \kappa}(-1)^\kappa f^{(\alpha)}(0)g^{(\kappa-\alpha)}(0)
$$
and call $f$ and $g$ {\em apolar} if they both have degree at most 
$\kappa$ and 
$\{f,g\}_\kappa=0$. 

Hinkkanen \cite{hink} wondered if Grace's theorem could be extended to 
several variables (he actually only considered multi-affine polynomials) but 
the precise form of such an extension remained uncertain. 
He also claimed that arguments due to Ruelle and 
Dyson \cite{ruelle1,ruelle4} could be extended to prove the following 
result. 

\begin{lemma}\label{l-wrong}
Let $A$ and $B$ be closed subsets of $\CC$ which do not contain the origin
 and let $f,g \in \CC_{(1^2)}[z_1,z_2]$. If $f$ is $A \times B$-stable and 
$g$ is $(\CC\setminus A) \times (\CC \setminus B)$-stable then 
$\{f,g\}_{(1^2)} \neq 0$. 
\end{lemma} 

Lemma \ref{l-wrong} is false, as one can see by 
considering for instance $f(z_1,z_2)=z_1+z_2$, $g(z_1,z_2)=1$, and 
$A=B=\{ \Im(z) \geq 1\}$. However, it holds under additional degree 
constraints (e.g.~if both $f$ ang $g$ have total degree $2$) which are tacitly 
assumed in \cite[Footnote 7]{ruelle4}.
In \cite{hink} Hinkkanen also proposed two possible generalizations 
of Grace's theorem as the following questions.

\begin{question}[Hinkkanen]
Let $A_i$, $i\in [n]$, be closed subsets of $\CC$ that do not 
contain the origin and $f,g \in \CC_{(1^n)}[z_1,\ldots,z_n]$. If $f$ is 
$A_1 \times \cdots \times A_n$-stable and $g$ is 
$(\CC\setminus A_1) \times \cdots \times (\CC\setminus A_n)$-stable then 
$\{f,g\}_{(1^n)} \neq 0$. 
\end{question} 

\begin{question}[Hinkkanen]
Let $C_i$, $i\in [n]$, be closed circular domains and 
$f,g \in \CC_{(1^n)}[z_1,\ldots,z_n]$. If $f$ is 
$C_1\times \cdots \times C_n$-stable and $g$ is 
$(\CC\setminus C_1) \times \cdots \times (\CC \setminus C_n)$-stable then 
$\{f,g\}_{(1^n)} \neq 0$. 
\end{question} 

We will now see that these questions are not true in full generality, but if 
we strengthen the hypothesis slightly in the second question then it is true 
for arbitrary degree polynomials. 

Note that if 
$f,g\in\CC_{\kappa}[z_1,\ldots,z_n]$ then
$$
\{f,g\}_{\kappa}(z):=\sum_{\alpha\le \kappa}(-1)^{\alpha}
f^{(\alpha)}(z)g^{(\kappa-\alpha)}(z)
$$
is a constant function so in this case $\{f,g\}_{\kappa}(z)=\{f,g\}_{\kappa}$ 
for $z\in \CC^n$. Elementary computations also yield the following.

\begin{lemma}\label{l-apro}
Let $\kappa=(\kappa_1,\ldots,\kappa_n)\in\NN^n$, 
$a_i,b_i,c_i,d_i\in \CC$, $a_id_i- b_ic_i=1$, $i\in [n]$, 
$f,g\in\CC_{\kappa}[z_1,\ldots,z_n]$ and set 
%\begin{itemize}
%\item[(1)] For $i\in [n]$ one has 
%$\dfrac{\partial}{\partial z_i}\{f,g\}_{\kappa}(z)=0$, $z\in\CC^n$.
%\item[(2)] Let 
\begin{equation*}
\begin{split}
F(z)&=(c_1z_1+d_1)^{\kappa_1}\cdots (c_nz_n+d_n)^{\kappa_n}
f\!\left(\frac{a_1z_1+b_1}{c_1z_1+d_1},\ldots,
\frac{a_nz_n+b_n}{c_nz_n+d_n}\right),\\
G(z)&= (c_1z_1+d_1)^{\kappa_1}\cdots (c_nz_n+d_n)^{\kappa_n}
g\!\left(\frac{a_1z_1+b_1}{c_1z_1+d_1},\ldots,
\frac{a_nz_n+b_n}{c_nz_n+d_n}\right).
\end{split}
\end{equation*}
Then $\{f,g\}_\kappa=\{F,G\}_\kappa$.
\end{lemma}

\begin{remark}\label{h-apol}
Note that Lemma \ref{l-apro} asserts that the functional 
$\{\cdot,\cdot\}_\kappa$ is 
invariant under the action of the 
group of M\"obius transformations normalized as in \eqref{mobius}. For $n=1$ 
this is quite well-known \cite{RS} and motivates the name  
``apolar invariant'' for $\{\cdot,\cdot\}_\kappa$ which is 
classically used in invariant theory, umbral calculus, and the theory of 
algebraic curves \cite{K-R,rota-lect,s-b}. 
\end{remark}

\begin{lemma}\label{case-D}
Let  $f,g \in 
\CC[z_1,\ldots, z_n]$ and suppose that $g$ has  degree
$\kappa\in \NN^n$. If $f$ is $\DD$-stable and $g$ is 
$\CC\setminus \DD$-stable then $\{f,g\}_\kappa \neq 0$.
\end{lemma}

\begin{proof}
 Let
$$
f(z)=\sum_{\alpha} a_\alpha z^\alpha \mbox{ and } 
g(z)=\sum_{\alpha} b_\alpha
z^\alpha. 
$$
Then $h(z):= \sum_{\alpha} (-1)^\alpha b_{\kappa - \alpha}
z^\alpha$ is $\overline{\DD}$-stable by Corollary \ref{inv}. By compactness 
there is $\epsilon >0$ such that 
$|h(z)|>\epsilon$ for $z \in \overline{\DD}^n$. This means that there is 
$\delta>0$ such that 
$\sum_{\alpha} (-1)^\alpha b_{\kappa - \alpha}(1+\delta)^{|\alpha|}
z^\alpha$ is $\DD$-stable. Then by 
applying Corollary \ref{master-comp} (c) to the 
$\DD$-stable polynomials 
$$
\sum_{\alpha} a_\alpha z^\alpha w^\alpha \mbox{ and } 
\sum_{\alpha} (-1)^\alpha b_{\kappa - \alpha}(1+\delta)^{|\alpha|}
z^\alpha w^\alpha
$$
we deduce that the polynomial
$$
F(z,w):=\sum_{\alpha} \frac{(-1)^\alpha a_\alpha 
b_{\kappa - \alpha}(1+\delta)^{|\alpha|}}{\binom \kappa \alpha}
z^\alpha w^\alpha 
$$
is $\DD$-stable or identically zero. The assumptions on $f$ and $g$ 
guarantee that 
$a_0b_{\kappa} \neq 0$, so 
$F$ is $\DD$-stable and
$$
\{f,g\}_\kappa
= \kappa!F\!\big( (1+\delta)^{-1/2}, \ldots, (1+\delta)^{-1/2}\big) \neq 0, 
$$
as claimed.
\end{proof}

%\begin{theorem}\label{appo}
%Let $\{C_j\}_{j=1}^n$ and $\{D_j\}_{j=1}^n$ be open circular domains such that $C_j' \cap D_j' =\emptyset$ for all $1 \leq j \leq n$ and let $f,g \in \CC[z_1, \ldots, z_n]$ be of degree $\kappa$. If $f$ is 
%$\{C_j\}_{j=1}^n$-stable and $g$ is $\{D_j\}_{j=1}^n$-stable then $\{f,g\} \neq 0$.
%\end{theorem}

We find it most natural to state two apolarity theorems: one for discs and 
exterior of discs (Theorem \ref{appo}) and one for half-planes 
(Theorem \ref{ap-halfplane}). 

\begin{theorem}\label{appo}
Let $C_i$, $i\in [n]$, be open  discs or exterior of discs and 
let $f,g \in \CC_\kappa[z_1, \ldots, z_n]$. Suppose that   
\begin{itemize}
\item[(i)] $f$ is 
$ C_1\times \cdots \times C_n$-stable and $\deg_{z_j}(f)=\kappa_j$ 
whenever $C_j$ is the exterior of a disk, and
\item[(ii)] $g$ is 
$(\CC\setminus C_1) \times \cdots \times (\CC \setminus C_n)$-stable and 
$\deg_{z_j}(f)=\kappa_j$ whenever $C_j$ is a disk. 
\end{itemize}
Then $\{f,g\}_\kappa \neq 0$.
\end{theorem}

\begin{proof} 
For $i\in [n]$ let $\Phi_i$ be a M\"obius transformation 
$(a_i z + b_i)/(c_iz+d_i)$ as in \eqref{mobius}  
for which $\Phi_i(\DD) = C_i$. By Lemma \ref{translate} the polynomial 
$$
F(z):= (c_1z_1+d_1)^{\kappa_1}\cdots (c_nz_n+d_n)^{\kappa_n} 
f(\Phi_1(z_1), \ldots, \Phi_n(z_n))
$$
is $\DD$-stable. 
Note that $-d_i/c_i \notin {\CC\setminus \DD}$ since none of the $C_i$'s is a 
half-plane. The polynomial 
$$
G(z):=  (c_1z_1+d_1)^{\kappa_1}\cdots (c_nz_n+d_n)^{\kappa_n}g(\Phi_1(z_1), 
\ldots, \Phi_n(z_n))
$$
is therefore $\CC\setminus\DD$-stable.  By  Lemma~\ref{translate} the degree 
of $g$ is precisely $\kappa$ so Lemma~\ref{case-D} applies and by 
Lemma \ref{l-apro} we get 
$
\{f,g\}_\kappa= \{F,G\}_\kappa \neq 0 
$.
\end{proof}

The {\em homogeneous part} of a polynomial $f \in \CC[z_1,\ldots, z_n]$ is 
the polynomial $f_H$ obtained by extracting the terms of maximum total degree, 
i.e., 
$$
f_H(z_1, \ldots, z_n)= \lim_{t \rightarrow \infty} 
t^{-d}f(tz_1, \ldots, tz_n), 
$$
where $d= \max\{|\alpha|: \partial^\alpha f / \partial z^\alpha \neq 0\}$. 

\begin{lemma}\label{hom-part}
Suppose $f \in \CC[z_1,\ldots, z_n]$ is $\bH_0$-stable and $i\in [n]$. Then 
$$
\left(\frac {\partial f} {\partial z_i}\right)_H =   
\frac {\partial f_H} {\partial z_i}.
$$
\end{lemma}
\begin{proof}
Suppose that $f$ has total degree $d$. Clearly, it is enough to prove that 
either ${\partial f}/ {\partial z_i}$ is identically zero  or  its total 
degree is $d-1$. 

Assume that ${\partial f} /{\partial z_i} \not\equiv 0$ and its 
total degree is $d' < d-1$. By Remark \ref{gl} the polynomial 
$f + z_{n+1}{\partial f} /{\partial z_i}$ is $\bH_0$-stable. Consider now the 
univariate polynomials 
\begin{equation*}
\begin{split}
&p(z)=f(z,\ldots, z)= f_H(1, \ldots,1)z^d + \cdots,  \\
&q(z)= \frac {\partial f} {\partial z_i}(z,\ldots, z)
= \left(\frac {\partial f} {\partial z_i}\right)_H(1, \ldots,1)z^{d'}+ \cdots,
\end{split} 
\end{equation*}
These polynomials are of degree $d$ and $d'$ respectively, since the leading 
coefficients are non-zero by Lemma \ref{phase}. Solving for $z_{n+1}$ in 
$p(z)+z_{n+1}q(z)=0$ we see that 
$\Im(p(z)/q(z)) \geq 0$ whenever $\Im(z) >0$. This is a contradiction since 
$p(z)/q(z)= C z^{d-d'}+ o(z^{d-d'})$ when $z \rightarrow \infty$, where 
$C \neq 0$ and $d-d' \geq 2$. 
\end{proof}

\begin{theorem}\label{ap-halfplane}
Let $C_1$ and $C_2$ be two open half-planes with non-empty intersection, 
$\kappa\in\NN^n$, and $f,g \in \CC_\kappa[z_1, \ldots, z_n]$. 
If $f$ is $C_1$-stable, $g$ is $C_2$-stable, and 
$\kappa \leq \alpha +\beta$ for some 
$\alpha \in \supp(f), \beta \in \supp(g)$, then 
$\{f,g\}_\kappa \neq 0$. 
\end{theorem}  

\begin{proof}
By an affine transformation of the variables we may assume that there is an 
$\epsilon >0$ 
such that $f(z-i\varepsilon)$ and $g(-z+i\varepsilon)$ are $\bH_0$-stable, 
where $\varepsilon=(\epsilon,\ldots, \epsilon)$. Then so are the 
$2n$-variable polynomials $f(z+w-i\varepsilon)$ and $g(-z-w+i\varepsilon)$ and 
by Corollary \ref{master-comp} also the polynomial 
$$
F(z,w)=\sum_{\alpha \leq \kappa}(-1)^{\kappa-\alpha}
f^{(\alpha)}(w-i\varepsilon)g^{(\kappa-\alpha)}(-z+i\varepsilon)
$$
unless it is identically zero. If it is not identically zero then the 
conclusion of the theorem follows by setting $z=w=i\varepsilon$. To complete 
the proof we show that $F(z,w)$ is not identically zero. Let 
$G_\alpha(z,w)=(-1)^{\kappa-\alpha}f^{(\alpha)}(w-i\varepsilon)
g^{(\kappa-\alpha)}(-z+i\varepsilon)$. This polynomial is $\bH_0$-stable or 
identically zero (by Remark \ref{gl}) and by Lemma \ref{hom-part} and 
Lemma \ref{phase} all non-zero coefficients in its homogeneous part have the 
same phase as those in the homogeneous part of 
$f(w-i\varepsilon)g(-z+i\varepsilon)$. By the assumptions on the supports of 
$f$ and $g$ there is an $\alpha$ such that $G_\alpha(z,w)\neq 0$ so 
$\lim_{t\rightarrow \infty} t^{-d-e+|\kappa|} F(tz,tw) \neq 0$, where $d$ and 
$e$ are the total degrees of 
$f$ and $g$, respectively. In particular, $F(z,w)$ is not identically zero.  
\end{proof}

%\begin{remark}\label{r-pap}
%Let $\kappa\in\NN^n$ and $f,g \in \CC_\kappa[z_1, \ldots, z_n]$. 
%Elementary computations yield
%$
%\{f,g\}_{\kappa}=\kappa!
%\big\{\Pi_\kappa^\uparrow(f),\Pi_\kappa^\uparrow(g)\big\}_{(1^{|\kappa|})}, 
%$
%where $\Pi_\kappa^\uparrow$ is the 
%polarization operator defined in Remark \ref{r-horm}. 
%In particular,  apolarity is preserved by polarization, which provides further 
%motivation for the classical terminology  ``apolar invariant'' 
%(cf.~Remark \ref{h-apol}). 
%\end{remark}

\section{Hard Lieb-Sokal Lemmas}\label{ss-91}
In \cite[Proposition 2.2]{LS} Lieb and Sokal proved that the operation that 
replaces one variable with differentiation with respect to another variable 
preserves weak Hurwitz stability. This result played a key role in the study 
of 
Laplace transforms of Lee-Yang measures and the extensions of Newman's 
strong Lee-Yang theorem obtained in \cite{LS}, see \S \ref{ss-newman}. 
It was also an essential ingredient in proving the 
sufficiency part of the classification theorems of \cite{BB-I}.

The Lieb-Sokal result is a 
``soft'' (transcendental/unbounded degree) result since it amounts to saying 
that the linear operator on $\CC[z_1, \ldots, z_n]$ acting on 
monomials as 
\begin{equation}\label{ls-soft}
z^{\alpha} \mapsto (-1)^{\alpha_1}
\cdot\frac{\partial^{\alpha_1}(z_2^{\alpha_2}z_3^{\alpha_3}\cdots 
z_n^{\alpha_n})}
{\partial z_2^{\alpha_1}},\quad 
\alpha=(\alpha_1,\ldots,\alpha_n)\in\NN^n,
\end{equation}
(which one may schematically represent as 
``$z_1\mapsto -\partial/\partial z_2$'') preserves ($\bH_0$-)stability, 
see Theorem \ref{th-LS} in \S \ref{ss-newman}.

By considering certain linear operators on finite-dimensional polynomial spaces
we can establish ``hard'' versions of Lieb-Sokal's result. 

\begin{lemma}\label{strong-LS}
Let $n,d \in \NN$ with $n \geq 2$ and let $\kappa \in \NN^n$ be such that 
$\kappa_1=\kappa_2=d$. Define a linear operator 
$T_d : \CC_\kappa[z_1,\ldots, z_n] \rightarrow \CC_\kappa[z_1,\ldots, z_n]$ by 
$$
T_d(f) =  \frac{1}{d!}\sum_{k=0}^d\frac{\partial^{d} f}
{\partial z_1^k \partial z_2^{d-k}}. 
$$
Then $T_d$ preserves $\bH_\te$-stability for any $0\le \te <2\pi$.
\end{lemma}

\begin{proof}
The symbol of $T_{d}$, i.e., the $2n$-variable polynomial 
$T_{d}\!\left[(z+w)^\kappa \right]$ is given by 
\begin{multline*}
\frac{1}{d!}(z_3+w_3)^{\kappa_3}\cdots (z_n+w_n)^{\kappa_n}
\sum_{k=0}^d \frac{d!}{(d-k)!} \frac{d!}{k!} 
(z_1+w_1)^{d-k}(z_2+w_2)^{k}\\
=(z_1+z_2+w_1+w_2)^d(z_3+w_3)^{\kappa_3}\cdots (z_n+w_n)^{\kappa_n}
\end{multline*}
which is $\bH_\te$-stable. The conclusion follows from 
Theorem \ref{open-halfplane}.
\end{proof}

\begin{remark}\label{rrr}
An interesting property of $T_{d}$ is that $T_{d}(f)$ is actually 
a polynomial in the $n-1$ variables $z_1+z_2, z_3, \ldots, z_n$ for any 
$f\in \CC_\kappa[z_1,\ldots, z_n]$.
Indeed, let $F(z_1,z_2,\ldots,z_n)=T_{d}(f)(z_1,z_2,\ldots,z_n)$. 
It is straightforward to show that 
$$
\frac{\partial}{\partial t}F(z_1+t,z_2-t,z_3,\ldots,z_n)=0,\quad t\in\CC,
$$
and by letting $t=z_2$ we get $F(z_1+t,z_2-t,\ldots,z_n)
=F(z_1+z_2,0,\ldots,z_n)$.
\end{remark}

Using Lemma \ref{strong-LS} and Remark \ref{rrr} we deduce the 
following ``hard'' result that 
substantially improves \eqref{ls-soft} when the top degree is specified.

\begin{corollary}\label{cor-strong}
Let $n,d \in \NN$ with $n \geq 2$ and let $\kappa \in \NN^n$ be such that 
$\kappa_1=\kappa_2=d$. Define linear operators 
$S_d, R_d : \CC_\kappa[z_1,\ldots, z_n] \rightarrow 
\CC_\kappa[z_1,\ldots, z_n]$ by
$$
S_{d}\!\left[\sum_{k=0}^{d}z_1^kQ_k(z_2,\ldots,z_n)\right] =
\frac{1}{d!}\sum_{k=0}^{d}k!\left(\frac{\partial}{\partial z_2}\right)^{d-k}
Q_k(z_2,\ldots, z_n)
$$
and
$$
R_{d}\!\left[\sum_{k=0}^{d}z_1^kQ_k(z_2,\ldots,z_n)\right]
=
\frac{1}{d!}\sum_{k=0}^{d}(-1)^{k}(d-k)!
\left(\frac{\partial}{\partial z_2}\right)^{k}
Q_k(z_2,\ldots, z_n).
$$
Then $S_{d}$ and $R_{d}$ preserve ($\bH_0$-)stability up to degree $\kappa$.
\end{corollary}

%\begin{proof}
%By arbitrarily fixing the variables $u_j,v_j$ with $2\le j\le n$ in 
%the upper half-plane $H$ we may assume wlog that
%$n=2$. Moreover, by Lemma \ref{l-close} (3) one has 
%$$
%\sum_{k=0}^{d}u_1^kQ_k(v_1)\in\HH_2(\CC)\Longleftrightarrow
%\sum_{k=0}^{d}(-1)^{k}u_1^{d-k}Q_k(v_1)\in\HH_2(\CC)
%$$
%and it is therefore enough to prove the assertion only for $S_{\bd}$. Now 
%if $f\in\CC[u_1,v_1]$ and $T_{\bd}$ is as in \eqref{td} then clearly
%$$
%S_{\bd}(f)(u_1,v_1)=T_{\bd}(f)(u_1,v_1)\big|_{u_1=0},
%$$
%so the result follows from Lemma \ref{strong-LS}.
%\end{proof}
 
The above ``hard'' results do indeed imply  the ``soft'' ones.
To see this fix 
$\beta\in\NN^n$ and set 
$(\beta)_{\alpha}=\alpha!\binom{\beta}{\beta-\alpha}$ for $\alpha\in\NN^n$. 
In \cite[Lemma 8.2]{BB-I} it was shown that the linear operator on 
$\CC[z_1,\ldots,z_n]$ defined by $z^\alpha\mapsto (\beta)_{\alpha}z^\alpha$, 
$\alpha\in\NN^n$, preserves stability. In particular, if
$
\sum_{k=0}^{d}z_1^kQ_k(z_2)$ is stable then 
$
\sum_{k=0}^{d}\frac{z_1^k}{(d-k)!}
Q_k(z_2)
$ is stable 
and this extends to $n$ variables. Therefore, the "soft" Lieb-Sokal result 
(respectively, Theorem~\ref{th-LS}) follows from Lemma \ref{strong-LS} 
(respectively, Corollary \ref{cor-strong}).

\section{Transcendental Symbols and the Weyl Algebra}\label{s-PDO}

Define the {\em complex Laguerre-P\'olya class} $\HHH_n(\CC)$ as 
the class of entire functions in $n$ variables that are limits, uniformly on 
compact sets, of polynomials in $\HH_n(\CC)$, see, e.g., \cite[Chap.~IX]{Le}. 
The usual (real) {\em Laguerre-P\'olya class} $\HHH_n(\RR)$ consists of all 
functions in $\HHH_n(\CC)$ with real coefficients.

If $T : \KK[z_1,\ldots, z_n] \rightarrow \KK[z_1,\ldots,z_n]$, where 
$\KK=\RR$ or 
$\CC$, is a linear operator 
we define its {\em transcendental symbol}, $\overline{G}_T(z,w)$, to be the 
formal power series in $w_1, \ldots, w_n$ with polynomial coefficients in 
$\KK[z_1,\ldots,z_n]$ given by 
$$
\overline{G}_T(z,w)
:=\sum_{\alpha \in \NN^n} (-1)^\alpha T(z^\alpha)
\frac {w^\alpha}{\alpha!}.
$$
By abuse of notation we write
$\overline{G}_T(z,w) =T[e^{-z\cdot w}]$, where 
$z\cdot w=z_1w_1+\ldots+z_nw_n$. Let us recall from \cite{BB-I} the 
transcendental characterizations of complex, respectively real 
stability preservers. 

\begin{theorem}\label{multi-infinite-stab}
Let $T : \CC[z_1,\ldots, z_n] \rightarrow \CC[z_1,\ldots, z_n]$ be a 
linear operator. Then $T$ preserves ($\bH_0$-)stability if and only if either
\begin{itemize}
\item[(a)] $T$ has range of dimension at most one and is given by 
$
T(f) = \alpha(f)P,
$
where $\alpha$ is a linear form on $\CC[z_1,\ldots, z_n]$ and 
$P\in\HH_n(\CC)$, or 
\item[(b)] $\overline{G}_T(z,w)\in\HHH_{2n}(\CC)$.
\end{itemize}
\end{theorem}

\begin{remark}\label{lhp}
From Theorem~\ref{multi-infinite-stab} one can easily deduce a 
characterization of linear operators preserving $\Omega$-stability for any 
open half-plane $\Omega$. For instance, the analog of 
Theorem~\ref{multi-infinite-stab} (b) for the open right half-plane 
$\bH_{\frac{\pi}{2}}$ is that the {\em transcendental symbol of} $T$ {\em 
with respect to} $\bH_{\frac{\pi}{2}}$, i.e., the formal power series 
$$
T[e^{z\cdot w}]:=\sum_{\alpha \in \NN^n}T(z^\alpha)\frac {w^\alpha}{\alpha!}, 
$$
defines an entire function which is the limit, uniformly on compact sets, of 
weakly Hurwitz stable polynomials. 
\end{remark}

\begin{theorem}\label{multi-infinite-hyp}
Let $T : \RR[z_1,\ldots, z_n] \rightarrow \RR[z_1,\ldots, z_n]$ be a linear 
operator. Then $T$ preserves real stability if and only if either
\begin{itemize}
\item[(a)] $T$ has at most $2$-dimensional range and is given by 
$
T(f) = \alpha(f)P + \beta(f)Q,
$
where $\alpha, \beta$ are real linear forms on $\RR[z_1,\ldots, z_n]$ and 
$P,Q\in\HH_n(\RR)$ are such that $P + iQ\in\HH_n(\CC)$, or 
\item[(b)] Either $\overline{G}_T(z,w)\in\HHH_{2n}(\RR)$ or 
$\overline{G}_T(z,-w)\in\HHH_{2n}(\RR)$.
\end{itemize}
\end{theorem}

To illustrate the power of Theorems \ref{multi-infinite-stab} and  
\ref{multi-infinite-hyp} we show that the main results of \cite{BBS3} 
for partial differential operators immediately follow from these two theorems. 
Recall that a (Weyl algebra) finite order linear partial differential 
operator with
polynomial coefficients is an operator 
$T: \KK[z_1,\ldots, z_n] \rightarrow \KK[z_1,\ldots, z_n]$, where $\KK=\CC$ 
or $\RR$, of the form 
\begin{equation}\label{difop}
T= \sum_{\alpha \leq \beta} Q_\alpha(z)
\frac{\partial^\alpha}{\partial z^\alpha},
\end{equation}
where $\beta \in\NN^n$ and $Q_\alpha\in \KK[z_1,\ldots, z_n]$, 
$\alpha\le \beta$.

\begin{theorem}\label{krull}
Let $T: \KK[z_1,\ldots, z_n] \rightarrow \KK[z_1,\ldots, z_n]$, where 
$\KK=\CC$ or $\RR$, be defined by \eqref{difop} and set 
$$
F(z,w) = \sum_{\alpha \leq \beta} Q_\alpha(z)w^\alpha \in 
\KK[z_1,\ldots z_n, w_1,\ldots, w_n]. 
$$
Then 
\begin{itemize}
\item[(a)] $T$ preserves stability if and only if $F(z,-w)$ is stable; 
\item[(b)] $T$ preserves real stability if and only if $F(z,-w)$ is real 
stable. 
\end{itemize}
\end{theorem}

\begin{proof}
The 
(transcendental) symbol of $T$ is given by 
$$
T[e^{-z\cdot w}]= e^{-z\cdot w} F(z,-w),
$$
so (a) and (b) follow immediately from 
Theorems \ref{multi-infinite-stab} and  \ref{multi-infinite-hyp}, 
respectively. 
\end{proof}

\begin{remark}
Theorem \ref{krull} was first established in 
\cite[Theorems 1.2--1.3]{BBS3} by different methods. An interesting
consequence noted in \cite[Theorem 1.11]{BBS3} is 
that if a Weyl
algebra operator $T$ preserves (real) stability then so does its Fischer-Fock
dual $T^*$. As shown in \cite{BBS3}, this duality result is a 
powerful multivariate generalization of the classical 
Hermite-Poulain-Jensen theorem and P\'olya's curve theorem \cite{CC1,RS}.
\end{remark}

\section*{{\bf B. Applications}}

We will now apply the theory developed in Part A to show that 
(the key steps in) existing proofs
and generalizations of the Lee-Yang and Heilmann-Lieb theorems follow in a 
simple and unified way from the characterizations of $\Omega$-stability 
preservers in terms of operator symbols obtained in \cite{BB-I}. 
These results are due to 
Asano \cite{As}, Ruelle \cite{ruelle-g1,ruelle-g2,ruelle5}, 
Newman \cite{new1,new2}, Lieb-Sokal \cite{LS}, Hinkkanen \cite{hink},
Choe et al \cite{COSW}, Wagner \cite{W2}. For brevity's sake, we will
only focus on the main arguments used in deriving them and in some cases we 
point out possible extensions.

\section{Recovering Lee-Yang and Heilmann-Lieb Type Theorems}\label{ss-94}

Let us first recall the original version of the Lee-Yang theorem for  
the {\em partition function} of the 
{\em ferromagnetic Ising model} (at inverse temperature $1$). This function
may be written as 
$$
Z(h_1,\ldots, h_n)=\sum_{\sigma \in \{-1,1\}^n}\mu(\sigma)e^{\sigma\cdot h}, 
$$
where $\sigma \cdot h = \sum_{i=1}^n \sigma_i h_i$ and 
$\mu(\sigma)= e^{\sum_{i,j =1}^n J_{ij}\sigma_i\sigma_j}$. 
 
\begin{theorem}[Lee-Yang \cite{LY}]\label{ly-orig}
If $J_{ij} \geq 0$ for all $i,j \in [n]$ then 
\begin{itemize}
\item[(a)] $Z(h_1,\ldots, h_n)\neq 0$ whenever $\Re(h_i)>0$, 
$1 \leq i \leq n$; 
\item[(b)] All zeros of $Z(h,\ldots, h)$ lie on the imaginary axis. 
\end{itemize}
\end{theorem}

\begin{remark}\label{r-phys}
In physical terms \cite{BBCK,BBCKK,LY,LS,sok}, the $J_{ij}$ are 
ferromagnetic ($\ge 0$) coupling 
constants while the $h_i$ are external (magnetic) fields sometimes also 
called fugacities. 
Theorem \ref{ly-orig} (b) asserts that the zeros of the 
partition function of the ferromagnetic Ising model accumulate on the 
imaginary axis in the complex fugacity plane and a (first-order) phase 
transition occurs only at zero magnetic field.
\end{remark}

Before we give a proof of the Lee-Yang theorem let us make a historical 
digression. In his work on the zeros of the Riemann zeta 
function P\'olya was led to a simple yet useful result: 

\begin{lemma}[P\'olya \cite{pol-riem}, Hilfssatz II]\label{pol-h2}
Let $a>0$, $b\in\RR$, and $G(z)$ be a real entire function of genus
$0$ or $1$ with at least one real zero and only real zeros. Then the function
\begin{equation}\label{polypoly}
G(z+ia)e^{ib}+G(z-ia)e^{-ib}
\end{equation}
has only real zeros.
\end{lemma}

Hilfssatz II was subsequently employed 
by Kac \cite[pp.~424--426]{pol-coll} to settle a special case
of Theorem \ref{ly-orig} that proved to be inspirational for Lee and
Yang's final proof \cite{LY} (cf.~\cite[Remark 4.2]{BB-I}). Recently, 
Lee-Yang type results and applications to Fourier transforms with all
real zeros were obtained in \cite{AC,car-F,car} by 
iterating the process of Hilfssatz II. A simple proof of this result 
and multivariate extensions is as follows. Let $R$ be the linear 
operator on formal power series in $n$ variables with complex coefficients 
$f(z)=\sum_{\alpha\in \NN^n}a(\alpha)z^\alpha$ defined by 
$$
R\!\left(\sum_{\alpha \in \NN^n} a(\alpha)z^\alpha\right)
= \sum_{\alpha \in \NN^n} \Re(a(\alpha))z^\alpha=\frac{1}{2}\!\left(f(z)
+\overline{f(\bar{z})}\right).
$$
By Lemma~\ref{krein} $R$ maps the set of stable polynomials into the set of 
real stable polynomials and 
consequently also the complex Laguerre-P\'olya class $\HHH_n(\CC)$ 
(cf.~\S \ref{s-PDO}) into the Laguerre-P\'olya class $\HHH_n(\RR)$. 
In the special case when $n=1$ and $G(z)$ is as in Lemma \ref{pol-h2} it 
follows from Hadamard's factorization theorem that $G(z)\in \HHH_1(\RR)$ 
hence $G(z+ia)e^{ib}\in \HHH_1(\CC)$ and by the above
$$
2R\!\left(G(z+ia)e^{ib}\right)= G(z+ia)e^{ib}+G(z-ia)e^{-ib}\in \HHH_1(\RR), 
$$
so the function in \eqref{polypoly} has only real zeros. Note also that 
$e^{-iz}\in \HHH_n(\CC)$ and thus $2R(e^{-iz})=\cos(z)\in \HHH_n(\RR)$. 

More general versions of Theorem \ref{ly-orig} were obtained in e.g.~\cite{LS}
and \cite{new2}, see \S \ref{ss-newman}. For simplicity of argument and 
exposition
we will concentrate for the moment just on the original Lee-Yang theorem 
and give a short proof based on the ideas in \cite{LS} combined with
Theorem \ref{multi-infinite-stab}.

\begin{proof}[Proof of Theorem \ref{ly-orig}]
Note that (b) follows from (a) by symmetry in $\sigma \mapsto -\sigma$. To
prove (a) define $\mathcal{M}$ to be the set of functions 
$\mu :  \{-1,1\}^n \rightarrow \CC$ whose Laplace transform 
$$
Z_\mu= \sum_{\sigma \in \{-1,1\}^n}\mu(\sigma)e^{\sigma\cdot h} 
$$
is the limit, uniformly on compact sets, of weakly Hurwitz stable polynomials
(i.e., non-vanishing whenever all variables are in the open right 
half-plane $\bH_{\frac{\pi}{2}}$).

\medskip

{\em Claim}:  Let $i,j\in [n]$ and $J_{ij}\geq 0$. If 
$\mu \in \mathcal{M}$ then $\tilde{\mu}_{ij} \in \mathcal{M}$, where 
$$
\tilde{\mu}_{ij}(\sigma)= 
\begin{cases}e^{J_{ij}}\mu(\sigma) \mbox{ if } \sigma_i=\sigma_j,  \\
  e^{-J_{ij}}\mu(\sigma) \mbox{ if } \sigma_i\neq\sigma_j.
\end{cases}
$$

Let us show that the claim implies the theorem. Indeed, 
if $\mu_0:\{-1,1\}^n \rightarrow \CC$  
is such that $\mu(\sigma)=1$ for all $\sigma \in \{-1,1\}^n$ then its
Laplace transform $Z_{\mu_0}$ equals 
$(e^{h_1}+e^{-h_1})\cdots (e^{h_n}+e^{-h_n})$. As noted above one has   
$\cos(z)\in \HHH_n(\RR)$, which implies that $\mu_0\in\mathcal{M}$ by a 
rotation of the variables. Then by successively applying to $\mu_0$ the 
transformations defined above for all pairs $(i,j)\in [n]\times [n]$ 
one gets (a).

To prove the claim note that $Z_{\tilde{\mu}_{ij}}=T(Z_{\mu_0})$, 
where 
$$
T= \cosh(J_{ij})
+ \sinh(J_{ij})\frac{\partial^2}{\partial z_i \partial z_j}. 
$$
By Theorem \ref{multi-infinite-stab} and Remark \ref{lhp} the  
operator $T$ preserves weak Hurwitz 
stability. Since $T$ is a second order (linear) differential operator, 
by standard results in complex analysis we have that if $f_k \rightarrow f$ 
uniformly on compacts then 
 $T(f_k) \rightarrow T(f)$ uniformly on compacts. This proves the claim.
\end{proof}

\subsection{Newman's Theorem and the Lieb-Sokal Approach}\label{ss-newman}

In \cite{new2} Newman proved a strong Lee-Yang theorem stating 
that the Lee-Yang property holds for one-component ferromagnetic pair 
interactions if and 
only if it holds for zero pair interactions. This theorem was subsequently
generalized in \cite{LS} by Lieb and Sokal who showed that one-component 
ferromagnetic pair interactions are ``universal multipliers for Lee-Yang
measures'' and established a similar result for two-component 
ferromagnets. 
Lieb-Sokal's key observation was that it would suffice to show that
a certain linear differential
operator preserves the Lee-Yang property, which they proved by reducing the
problem to the following statement about polynomials. 
 
\begin{theorem}[Lieb-Sokal]\label{th-LS}
Let $\{P_i(u)\}_{i=1}^m$ and $\{Q_i(v)\}_{i=1}^m$ be polynomials in $n$ 
complex variables $u=(u_1,\ldots,u_n)$, $v=(v_1,\ldots,v_n)$, and define 
\begin{equation*}
\begin{split}
R(u,v)&= \sum_{i=1}^m P_i(u)Q_i(v), \\
S(z)&= \sum_{i=1}^m P_i(\partial/\partial z)Q_i(z),
\end{split}
\end{equation*}
where $z=(z_1,\ldots,z_n)$, 
$\partial/\partial z=(\partial/\partial z_1,\ldots,\partial/\partial z_n)$. 
If $R$ is weakly Hurwitz stable (in $2n$ variables) then $S$ is either weakly 
Hurwitz stable or identically zero. 
\end{theorem}

\begin{proof}
Define a linear operator 
$$
T: \CC[u_1,\ldots, u_n, v_1,\ldots, v_n] 
\rightarrow \CC[u_1,\ldots, u_n, v_1,\ldots, v_n]
$$ 
by letting 
$$
T(u^\alpha v^\beta)= \frac{\partial^{\alpha}}
{\partial v_1^{\alpha_1}  \cdots \partial v_n^{\alpha_n}}\big(v^\beta\big), 
\quad \alpha, \beta \in \NN^n, 
$$
and extending linearly. Clearly, the theorem is equivalent to proving that 
$T$ preserves weak Hurwitz stability. By Theorem \ref{multi-infinite-stab} and 
Remark \ref{lhp} this amounts 
to showing that the formal power series
$$
\overline{G}_T(u,v,\xi,\eta)=\sum_{\alpha, \beta}T(u^\alpha v^\beta) 
\frac{\xi^\alpha \eta^\beta}{\alpha!\beta!}
$$
(i.e., the transcendental symbol for $\bH_{\frac{\pi}{2}}$) defines an entire 
function which is the limit, uniformly on compact sets, of weakly Hurwitz 
stable polynomials. An elementary computation then yields
$$
\overline{G}_T(u,v,\xi,\eta) =\prod_{i=1}^n
\big(e^{\eta_iv_i}e^{\eta_i\xi_i}\big),
$$
which satisfies the above requirement 
since $e^{zw}= \lim_{n \rightarrow \infty}(1+zw/n)^n$. 
\end{proof}

\subsection{The Schur-Hadamard Product and Convolution}

The following version of Theorem \ref{ly-orig} is usually referred to as 
the Lee-Yang ``circle theorem'', see, e.g., \cite{hink,ruelle5}.

\begin{theorem}\label{lee-yang} 
Let $A=(a_{ij})$ be a 
Hermitian $n\times n$ matrix whose entries are in the closed unit 
disk $\overline{\DD}$. Then the polynomial 
$$
f(z_1,\ldots, z_n)= \sum_{S \subseteq [n]} 
z^S\prod_{i \in S}\prod_{j \notin S}a_{ij} 
$$
is $\DD$-stable. In particular, $f(z, \ldots,z)$ has all its zeros on the 
unit circle. 
\end{theorem} 

Hinkkanen's proof \cite{hink} of Theorem \ref{lee-yang} makes use of a 
composition theorem for the {\em Schur-Hadamard product} of 
multi-affine polynomials which is defined as follows: if 
$f(z)=\sum_{S \subseteq [n]} a(S)z^S$ and 
$g(z)=\sum_{S \subseteq [n]} b(S)z^S$ then 
$$
(f\bullet g)(z) = \sum_{S \subseteq [n]} a(S)b(S)z^S. 
$$

The next result is Hinkkanen's composition theorem \cite[Theorem C]{hink}.

\begin{theorem}\label{hct}
Let $f,g \in \CC_{(1^n)}[z_1, \ldots, z_n]$. If $f,g$ are $\DD$-stable 
then so is $f \bullet g$ unless it is identically zero. 
\end{theorem}

\begin{proof}
Let $g$ be a fixed $\DD$-stable multi-affine polynomial in $n$ variables and 
let $T$ be the linear transformation on multi-affine polynomials in $n$ 
variables given by $T(f)=f \bullet g$. Recall Theorem \ref{open-disk} (b). 
The symbol of $T$ is 
$$
T\!\left[(1+zw)^{[n]}\right]=g(z_1w_1, \ldots, z_nw_n),
$$ 
which is clearly $\DD$-stable (in $2n$ variables). 
Theorem~\ref{open-disk} yields the result.
\end{proof}

\begin{remark}\label{hink-proof}
The proof of Theorem \ref{lee-yang} given in \cite{hink} is as follows.  
For $i,j\in [n]$ with $i<j$ (note the typo ``$i\neq j$'' in \cite{hink}) let 
$$
f_{ij}(z_1,\ldots,z_n)=(1+a_{ij}z_i + \overline{a_{ij}}z_j+z_i z_j)
\prod_{k \in [n]\setminus \{i,j\}}(1+z_k). 
$$
It is not hard to see that $f_{ij}$ is $\DD$-stable and by taking the 
Schur-Hadamard product of all these polynomials one gets
$$
(f_{12} \bullet \cdots \bullet f_{(n-1) n})(z)= \sum_{S \subseteq [n]} 
z^S\prod_{i \in S}\prod_{j \notin S}a_{ij}, 
$$
which is again $\DD$-stable by Theorem~\ref{hct}.
\end{remark}

Using Corollary \ref{master-comp} we can extend Hinkkanen's 
composition theorem to arbitrary (not necessarily multi-affine) $\DD$-stable 
polynomials:

\begin{theorem}\label{g-compo}
Let $f(z)=\sum_{\alpha \leq \kappa}\binom \kappa \alpha a(\alpha)z^\alpha$ 
and $g(z)=\sum_{\alpha \leq \kappa}\binom \kappa \alpha b(\alpha)z^\alpha$ be 
$\DD$-stable polynomials. Then so is 
$$
(f\bullet g)(z):=
\sum_{\alpha \leq \kappa}\binom \kappa \alpha a(\alpha)b(\alpha)z^\alpha 
$$
unless it is identically zero.
\end{theorem}

\begin{proof}
Apply Corollary \ref{master-comp} (c) 
to the $\DD$-stable polynomials $f(z)$ and 
$g(zw)$, where $zw=(z_1w_1,\ldots,z_nw_n)$. 
\end{proof}

Closely related to the Schur-Hadamard product is the {\em convolution 
operator} on multi-affine polynomials \cite{COSW} defined as follows: if 
$f(z)=\sum_{S \subseteq [n]} a(S)z^S$ and 
$g(z)=\sum_{S \subseteq [n]} b(S)z^S$ then 
$$
(f\star g)(z) = \sum_{S,T \subseteq [n]} a(S)b(T)z^{S\Delta T}, 
$$ 
where $S\Delta T=(S\cup T)\setminus (S\cap T)$. A corresponding composition 
result -- this time for weak Hurwitz stability -- is given in 
\cite[Proposition 4.20]{COSW}.

\begin{theorem}\label{t-conv}
Let $f,g \in \CC_{(1^n)}[z_1, \ldots, z_n]$. If $f,g$ are 
weakly Hurwitz stable then so is $f \star g$  
unless it is identically zero. 
\end{theorem}

\begin{proof}
Let $g$ be a fixed $\bH_{\frac{\pi}{2}}$-stable multi-affine polynomial in 
$n$ variables and 
let $T$ be the linear transformation on multi-affine polynomials in $n$ 
variables given by $T(f)=f \star g$. The symbol of $T$ 
(cf.~Theorem \ref{open-halfplane}) is just 
\begin{equation}\label{eq-conv}
i^nT\!\left[(z+w)^{[n]}\right]=(z+w)^{[n]}
g\!\left(\frac{1+z_1w_1}{z_1+w_1}, \ldots,\frac{1+z_nw_n}{z_n+w_n}\right).
\end{equation}
Now if $u,v\in \bH_{\frac{\pi}{2}}$ then also $u^{-1},v^{-1},u+v\in 
\bH_{\frac{\pi}{2}}$ hence 
$$
\frac{1+uv}{u+v}=(u+v)^{-1}+\left(u^{-1}+v^{-1}\right)^{-1}\in 
\bH_{\frac{\pi}{2}},
$$
so that the polynomial in \eqref{eq-conv} is $\bH_{\frac{\pi}{2}}$-stable 
(in $2n$ variables). 
Theorem~\ref{open-halfplane} again yields the desired conclusion.
\end{proof}

\subsection{Asano Contractions}\label{ss-asano}

Many known proofs of the Lee-Yang theorem are based on so-called 
{\em Asano contractions} or variations thereof 
\cite{As,ruelle3,ruelle4,ruelle5}. Let 
$$
f(z_1,\ldots, z_n)=a(z_3,\ldots, z_n)+ b(z_3,\ldots, z_n)z_1 
+c(z_3,\ldots, z_n)z_2+d(z_3,\ldots, z_n)z_1z_2
$$
be a polynomial in $n\ge 2$ variables which is multi-affine in $z_1$ and 
$z_2$. The Asano contraction of $f$ is 
$$
A(f)(z_1,\ldots, z_n)= a(z_3,\ldots, z_n)+d(z_3,\ldots, z_n)z_1.
$$
Note that $A(f)$ does not depend on $z_2$. The key fact used in the 
aforementioned proofs is a property of Asano 
contractions that may be stated as follows.

\begin{lemma}\label{l-asano}
Let $\kappa=(\kappa_1,\ldots,\kappa_n)\in \NN^n$ with $n\ge 2$ and
$\kappa_1=\kappa_2=1$. Then 
$$
A : \CC_\kappa[z_1,\ldots, z_n] \rightarrow \CC_\kappa[z_1, \ldots, z_n]
$$
is a linear operator that preserves $\DD$-stability.
\end{lemma}

\begin{proof}
It is clear that $A$ is linear. Its (algebraic) symbol is 
$$
A[(1+zw)^\kappa]= (1+zw)^{(\kappa_3,\ldots,\kappa_n)}(1+z_1w_1w_2), 
$$
which is $\DD$-stable, so the assertion follows from 
Theorem~\ref{open-disk}. 
\end{proof}

\subsection{Multi-Affine Part and Folding mod $2$}\label{ss-map}
Recall that a matching in a graph $G=(V,E)$ is a subset $M$ of $E$ such that 
no vertex of the graph $(V,M)$ has 
 degree exceeding one. The general version of the 
Heilmann-Lieb theorem on 
the {\em monomer-dimer model} is the following. 

\begin{theorem}[Heilmann-Lieb \cite{HL}]\label{heil-lie}
Let $G=(V,E)$ be a loopless graph and define its matching polynomial with
edge weights $\{\lambda_e\}_{e\in E}$ and vertex weights $\{z_i\}_{i\in V}$ as
$$
M_G(z,\lambda)=\sum_{\text{matchings $M$}}\prod_{e=ij\in M}\lambda_ez_iz_j.
$$
If $\lambda_e\ge 0$, $e\in E$, then $M_G(z,\lambda)$ is a weakly Hurwitz 
stable polynomial (in $z$).
\end{theorem}

In \cite{COSW} and \cite[\S 5]{sok} it was shown that the Lee-Yang 
and Heilmann-Lieb theorems can actually be given a unified combinatorial 
formulation and proof. The idea is to form the ``test''
polynomial 
\begin{equation}\label{test-p}
F_G(z,\lambda)=
 \prod_{e=\{i,j\} \in E} (1+\lambda_e z_i z_j)
\end{equation}
associated to a graph $G=(V,E)$, $|V|=n$, equipped with vertex weights 
$\{z_i\}_{i\in V}$ and non-negative edge 
weights $\{\lambda_e\}_{e\in E}$. This 
polynomial is weakly Hurwitz stable in the $z_i$'s 
and by applying to it appropriate 
linear operators one gets precisely the polynomials occurring in the 
Lee-Yang theorem and the Heilmann-Lieb theorem, respectively. Thus one only
has to check that the linear operators used in this process 
preserve weak Hurwitz stability. These operators are defined as 
follows.

The linear operator $\MAP:\CC[z_1,\ldots,z_n]\to\CC_{(1^n)}[z_1,\ldots,z_n]$  
extracts the multi-affine part of a polynomial, that is, if 
$f(z) = \sum_{\alpha \in \NN^n} a(\alpha)z^\alpha$ then  
$$
\MAP(f)(z)=\sum_{\alpha:\,\alpha_i \leq 1,\,i\in [n]} a(\alpha) z^\alpha. 
$$
The transcendental symbol (see Remark \ref{lhp}) of $\MAP$ is 
$$
\sum_{\alpha:\,\alpha_i \leq 1,\,i\in [n]} 
 z^\alpha \frac {w^\alpha}{\alpha!}=(1+zw)^{[n]}. 
$$
Clearly, this is a weakly Hurwitz stable polynomial. 
Theorem \ref{multi-infinite-stab} and Remark \ref{lhp}
imply that $\MAP$ preserves weak Hurwitz stability. It is easy to see that 
\begin{equation}\label{eq-map}
M_G(z,\lambda)= \MAP\Big[F_G(z,\lambda)\Big], 
\end{equation}
which yields the Heilmann-Lieb theorem. 

%\begin{remark}\label{r-sqrt}
%Suppose that $\lambda_e=1$, $e\in E$, set for convenience 
%$V=\{v_i\}_{i=1}^n$ and let $\kappa=(\kappa_1,\ldots,\kappa_n)\in\NN^n$. Since 
%$$
%\MAP[(1+zw)^\kappa]=\prod_{i=1}^n(1+\kappa_iz_iw_i),
%$$
%by rescaling the variables it follows from \eqref{eq-map} and 
%Theorem \ref{open-disk}  that 
%$M_G(z,1)\neq 0$ whenever $|z_i|<\deg(v_i)^{-1/2}$, $i\in [n]$.
%\end{remark}

The linear operator $\MOD:\CC[z_1,\ldots,z_n]\to\CC_{(1^n)}[z_1,\ldots,z_n]$ 
``folds mod 2'' the powers in the Taylor expansion
of a polynomial, i.e.,  if 
$f(z) = \sum_{\alpha \in \NN^n} a(\alpha)z^\alpha$ then  
$$
\MOD(f)(z)=\sum_{\alpha \in \NN^n} a(\alpha)z^{\alpha \bmod 2},
$$
where $\alpha \bmod 2=(\alpha_1 \bmod 2,\dots,\alpha_n \bmod 2)$. 
The algebraic symbol of $\MOD$ (up to degree $\kappa$) with respect to the
open right-half plane $\bH_{\frac{\pi}{2}}$ (cf.~Theorem \ref{open-disk}) is
$$
\MOD[(1+zw)^\kappa]
=2^{-n}(1+w)^\kappa(1+z)^{[n]}\prod_{i=1}^{n}\!\left[ 1+ 
\left( \frac {1-w_i}{1+w_i}\right)^{\kappa_i} 
\frac {1-z_i}{1+z_i} \right].
$$
If $\Re(\zeta)>0$ then $|(1-\zeta)/(1+\zeta)| < 1$ so the above symbol is 
weakly Hurwitz stable. By Theorem \ref{open-disk} we conclude that 
$\MOD$ preserves weak Hurwitz stability. $\MOD$ is employed in 
\cite[\S 4.8]{COSW} to prove the Asano contraction
lemma (Lemma \ref{l-asano}) and thereby the Lee-Yang theorem as well 
(cf.~\S \ref{ss-asano} and \cite[\S 5]{sok}). 

%\begin{remark}\label{other-lin}
%In \cite[\S 4]{COSW} several other linear operators on multivariate
%polynomials are shown to preserve
%the ``half-plane property'' (weak Hurwitz stability), such as deletion,
%contraction, duality, parallel connection, series 
%connection, $2$-sum, principal (co)truncations and principal (co)extensions.
%We leave it as an exercise to the interested reader to derive quick proofs of
%these results by using the characterizations of stability preservers in terms
%of operator symbols from \cite{BB-I}.
%\end{remark}

\subsection{Generalizations of the Heilmann-Lieb Theorem}\label{ss-trunc}

It is natural to study graph polynomials with more 
general degree constraints than those defining the matching polynomial
and to establish Heilmann-Lieb type theorems for such polynomials.
This has been pursued by 
e.g.~Ruelle \cite{ruelle-g1,ruelle-g2} and Wagner \cite{W1,W2}. 
If $G=(V,E)$ is a 
graph with vertex set $V=\{v_1,\ldots, v_n\}$ we let $\deg G \in \NN^n$ be 
the {\em degree vector} of $G$, i.e., the $i$-th coordinate of $\deg G$ is 
the degree of $v_i$ in $G$. Let $\kappa=(\kappa_1,\ldots,\kappa_n)\in\NN^n$ and
suppose that $\deg G \leq \kappa$. Then given degree weights 
 $u : \NN^\kappa \rightarrow \CC$ and non-negative edge 
weights $\{\lambda_e\}_{e\in E}$ one may ask what are the
non-vanishing properties of the polynomial 
 \begin{equation}\label{mastergraph}
 F_G(z,\lambda, u) = \sum_{H \subseteq E}\lambda^H u(\deg(V,H))z^{\deg(V,H)}. 
 \end{equation}
This question was considered by Wagner 
in \cite{W2}. When it comes to weak Hurwitz stability it is natural 
(and of course sufficient) to require that the linear ``truncation'' operator 
$
T: \CC_\kappa[z_1,\ldots, z_n] \rightarrow \CC_\kappa[z_1,\ldots, z_n]
$ 
defined by 
 $T(z^\alpha)=u(\alpha)z^\alpha$, $u(\alpha)\in\CC$, $\alpha\le \kappa$, 
preserves weak Hurwitz stability. Now these are precisely the multivariate
multiplier sequences (up to degree $\kappa$) that 
were characterized in  
Corollary \ref{multi-Hurwitz} of this 
paper as follows:  
 $$
 u(\alpha) = u_1(\alpha_1)\cdots u_n(\alpha_n), \quad \alpha\le \kappa,
 $$
 where for each $i\in [n]$ the polynomial 
\begin{equation}\label{uik}
\sum_{k=0}^{\kappa_i} \binom  {\kappa_i} k u_i(k)z^k 
\end{equation}
has all real non-positive zeros. We thus 
recover the following generalization of the Heilmann-Lieb theorem due to 
Wagner \cite{W2}, which extends a theorem of Ruelle \cite{ruelle-g1}. By the 
necessity in Corollary \ref{multi-Hurwitz} the theorem below is optimal. 

\begin{theorem}[Wagner \cite{W2}]\label{trunc-w}
Let $G=(V,E)$ be a graph whose degree vector satisfies 
$\deg G \leq \kappa$ and let 
$F_G(z,\lambda, u)$ and $u$ be as in \eqref{mastergraph} and \eqref{uik}, 
respectively. If $\{\lambda_e\}_{e\in E}$  are non-negative edge 
weights  then 
\begin{itemize}
\item[(a)] $F_G(z,\lambda, u)$ is weakly Hurwitz stable considered as a 
polynomial in $z$; 
\item[(b)] All zeros of the univariate polynomial 
$$
\sum_{k=0}^{|E|} \left(\sum_{\stackrel{H \subseteq E}
{ \mbox{\tiny{$|H|=k$}}}} u(\deg(V,H)) \lambda^H\right) t^k
$$
are real and non-positive.
\end{itemize}
\end{theorem}

\begin{proof}
The first statement follows from Corollary \ref{multi-Hurwitz} and the fact 
that the test polynomial defined in \eqref{test-p} is weakly Hurwitz stable. 
If we set all the $z_j$'s, $1\leq j \leq n$, equal to $-it$ we obtain the 
univariate polynomial (in $t$)  
$$
\sum_{k=0}^{|E|} \left(\sum_{\stackrel{H \subseteq E}
{ \mbox{\tiny{$|H|=k$}}}} u(\deg(V,H)) \lambda^H\right) (-1)^kt^{2k}
$$
which is then real stable. Clearly, this forces the polynomial in (b) to have 
all zeros real and 
non-positive. 
\end{proof}
 
\begin{remark}
In \cite{W2} Wagner also proves non-vanishing properties in 
sectors, which cannot be obtained by our methods. However, 
Theorem \ref{trunc-w} is slightly 
more general in that we consider max-degree at every vertex (not uniform 
max-degree).
\end{remark}

\section{Appendix}\label{app}
We give here a simple proof of the fact  that $\Sym$ can be viewed as a 
(convergent) infinite 
product of 
operations as those in Proposition~\ref{asymsym}. 
For $\sigma \in \sym_n$  define an operator 
$T_\sigma : \CC[z_1,\ldots, z_n] \rightarrow \CC[z_1,\ldots, z_n]$ by 
$$
T_\sigma(f) = \frac 1 {|\langle \sigma \rangle|} 
\sum_{\tau \in \langle \sigma \rangle} \tau(f),
$$
 where $\langle \sigma \rangle$ is the subgroup of $\sym_n$ 
generated by $\sigma$. Given $\alpha, \beta \in \NN^n$ we write 
$\alpha \sim \beta$ if $\alpha$ is a rearrangement of $\beta$. If 
 $f(z) = \sum_{\alpha \in \NN^n}a(\alpha)z^\alpha \in \RR[z_1,\ldots, z_n]$ 
let the {\em symmetry index} of $f$ be defined by 
 $
 \si(f)= \sum_{\alpha \sim \beta} |a(\alpha) -a(\beta)|. 
 $
For $f = g+ih \in \CC[z_1,\ldots, z_n]$ with $g,h \in  \RR[z_1,\ldots, z_n]$ 
we define its symmetry index as $\si(f)=\si(g)+\si(h)$. 
 Hence $\si(f)=0$ if and only if $f$ is symmetric. 

{\allowdisplaybreaks
\begin{lemma}\label{index}
Let $\sigma \in \sym_n$ and $f(z) 
= \sum_\alpha a(\alpha)z^\alpha \in \CC[z_1,\ldots, z_n]$. Then 
\begin{equation}\label{si}
\si(T_\sigma(f)) \leq \si(f) 
\end{equation}
with equality if and only if $a(\sigma(\alpha))=a(\alpha)$ for all 
$\alpha \in \NN^n$, i.e., $T_\sigma(f)=f$. 
\end{lemma} 

\begin{proof}
We may assume that $f \in \RR[z_1,\ldots,z_n]$. 
Since $\si(\tau(f))=\si(f)$ for all $\tau \in \sym_n$ we have by the 
triangle inequality 
\begin{eqnarray*}
\si(T_\sigma(f)) &=& \sum_{\alpha \sim \beta} 
\frac 1 {|\langle \sigma \rangle|} 
\Big| \sum_{\tau \in \langle \sigma \rangle} a(\tau(\alpha))
-a(\tau(\beta))\Big| \\
&\leq& \sum_{\alpha \sim \beta} \frac 1 {|\langle \sigma \rangle|} 
\sum_{\tau \in \langle \sigma \rangle} |a(\tau(\alpha))-a(\tau(\beta))| \\
&=& \frac 1 {|\langle \sigma \rangle|} 
\sum_{\tau \in \langle \sigma \rangle} \si(\tau(f)) \\ &=& \si(f) 
\end{eqnarray*}
with equality if and only if the following condition holds: 
\begin{equation*}
\text{If $\alpha \sim \beta$ then  $a(\tau(\alpha))-a(\tau(\beta))$ have 
the same sign for all $\tau \in \langle \sigma \rangle$.}\tag{A} 
\end{equation*}
Clearly, if $a(\sigma(\alpha))=a(\alpha)$ for all $\alpha \in \NN^n$ then 
equality holds in \eqref{si}. On the other hand,  
if equality in \eqref{si} holds let $\beta=\sigma(\alpha)$ and assume that 
$a(\alpha) \geq a(\beta)$ (the case $a(\alpha) \leq a(\beta)$ follows 
similarly). Then  by (A) we have 
$$
a(\alpha) \geq a(\sigma(\alpha)),    \quad a(\sigma(\alpha)) \geq 
a(\sigma^2(\alpha)),   \quad \ldots,  \quad  a(\sigma^{k-1}(\alpha)) \geq 
a(\alpha), 
$$ 
where $k$ is the order of $\sigma$. Hence $\alpha \mapsto a(\alpha)$ is 
constant on $\langle \sigma \rangle$-orbits, which completes the proof.
\end{proof}}

For $f \in \CC[z_1,\ldots, z_n]$ let 
$$
\Trp(f) =\{ T_{\tau_k} \cdots T_{\tau_1} (f) : 
\tau_1, \ldots, \tau_k \in \sym_n \mbox{ are transpositions}\}   
$$
and denote by $\overline{\Trp}(f)$ the set of polynomials that are limits, 
uniformly on compact sets, of polynomials in $\Trp(f)$. 

\begin{lemma}\label{sym-lim}
If $f \in \CC[z_1,\ldots, z_n]$ then  $\Sym(f) \in \overline{\Trp}(f)$. 
\end{lemma}

\begin{proof}
We claim that the set $\si\!\left(\overline{\Trp}(f)\right):=
\{\si(g): g \in  \overline{\Trp}(f) \}$ is closed. 
Suppose that $x_k \rightarrow x$ as $k\to \infty$, where $x_k= \si(g_k)$ with 
$g_k \in \overline{\Trp}(f)$ for $k\in\NN$.  
Let $|\cdot|_r$ be the supremum norm on the ball of radius $r$ in $\CC^n$. 
If $\sigma \in \sym_n$ we have by the triangle inequality and invariance 
under permutations that 
$|T_\sigma(g)|_r \leq  |g|_r$. It follows that $|h|_r \leq |f|_r$ for all 
$h \in  \overline{\Trp}(f)$. Hence, by Montel's theorem,  
$\{g_k\}_{k\in\NN}$ forms a normal family so there is a subsequence 
converging uniformly on compacts 
to a polynomial $g \in \overline{\Trp}(f)$ with $\si(g) =x$. 

Hence 
$
y:=\inf \si\!\left(\overline{\Trp}(f)\right)
$
is achieved for some 
$g \in \overline{\Trp}(f)$. If $\tau(g) \neq g$ for some transposition
 $\tau \in \sym_n$ then by Lemma~\ref{index} we have  
$\si(T_\tau(g)) < \si(g)$. However, one clearly has 
$T_\tau(g) \in \overline{\Trp}(f)$, which contradicts the minimality of 
$\si(g)$. We deduce that $\tau(g)=g$ for all transpositions 
$\tau \in \sym_n$ and thus $g = \Sym(f)$ and $\si(g)=0$.
\end{proof}

\section*{Acknowledgments}

We wish to thank the Isaac Newton Institute for Mathematical Sciences, 
University of Cambridge, for generous support during the programme  
``Combinatorics and Statistical Mechanics'' (January--June 2008), 
where this work, \cite{BB-I} and \cite{BBL} 
were presented in a series of lectures. Special thanks are due 
to the organizers and the participants in this programme for 
numerous stimulating discussions. 
We would also like to thank Alex Scott, 
B\'ela Bollob\'as and Richard Kenyon for the opportunities to announce 
these results in June--July 2008 through talks given at the 
Mathematical Institute, University of Oxford, the Centre for Mathematical 
Sciences, University of Cambridge, and the programme ``Recent Progress in 
Two-Dimensional Statistical Mechanics'', Banff International 
Research Station, respectively.

\end{document}